\documentclass[a4paper]{amsart}
\usepackage[utf8]{inputenc}
\usepackage[left=2.5cm,right=2.5cm,top=3cm,bottom=3.5cm]{geometry}

\usepackage{microtype} 

\usepackage{amsmath,amssymb}
\usepackage{enumitem}
\setlist[enumerate]{label=\normalfont{(\roman*)}}

\usepackage{aliascnt}
\usepackage{graphicx}
\usepackage[all,cmtip]{xy}

\usepackage[pdfusetitle,hidelinks]{hyperref}
\usepackage{amsthm}

\newcommand*{\mc}[1]{\mathcal{#1}}
\newcommand*{\opname}[1]{\operatorname{#1}}

\newcommand*{\GL}{\opname{GL}}
\newcommand*{\SL}{\opname{SL}}
\newcommand*{\SO}{\opname{SO}}
\newcommand*{\Sp}{\opname{Sp}}

\newcommand*{\PSL}{\opname{PSL}}
\renewcommand*{\phi}{\varphi}

\newcommand*{\NN}{\mathbb{N}}
\newcommand*{\ZZ}{\mathbb{Z}}
\newcommand*{\CC}{\mathbb{C}}
\newcommand*{\QQ}{\mathbb{Q}}
\newcommand*{\RR}{\mathbb{R}}
\newcommand*{\EE}{\mathbb{E}}
\newcommand*{\PP}{\mathbb{P}}

\newcommand*{\Jac}{\opname{Jac}}
\newcommand*{\ka}{\kappa}
\newcommand*{\la}{\lambda}
\newcommand*{\ga}{\gamma}

\newcommand*{\rk}{\opname{rk}}
\newcommand*{\pr}{\opname{pr}}
\newcommand*{\id}{\opname{id}}
\newcommand*{\tup}[1]{\underline{#1}}
\newcommand*{\gr}{{\opname{gr}}}
\newcommand*{\len}{\opname{l}}
\newcommand*{\Hom}{\opname{Hom}}
\newcommand*{\End}{\opname{End}}
\newcommand*{\Iso}{\opname{Iso}}
\newcommand*{\Aut}{\opname{Aut}}
\newcommand*{\im}{\opname{im}}
\newcommand*{\ev}{\opname{ev}}
\newcommand*{\Fl}{\opname{Fl}}
\newcommand*{\St}{\opname{St}}
\newcommand*{\norm}[1]{\|#1\|}

\newcommand*{\Sym}{\opname{Sym}}
\newcommand*{\Proj}{\opname{Proj}}
\newcommand*{\pardeg}{\opname{pardeg}}

\newcommand*{\Spec}{\opname{Spec}}
\newcommand*{\sstable}{\textup{(}semi\nobreakdash-\textup{)}stable}
\newcommand*{\PB}{P}
\newcommand*{\rad}{\mc{R}}
\newcommand*{\QTmp}{\opname{QTmp}}
\newcommand*{\Gies}{\opname{Gies}}
\newcommand*{\gies}{\opname{gies}}
\newcommand*{\Quot}{\opname{Quot}}
\newcommand*{\QPPB}{\opname{QPPB}}
\newcommand*{\df}{\opname{df}}
\newcommand*{\ad}{\opname{ad}}

\newcommand{\mynewtheorem}[2]{
  \newaliascnt{#1}{dummy}
  \newtheorem{#1}[#1]{#2}
  \aliascntresetthe{#1}
  \expandafter\def\csname #1autorefname\endcsname{#2}
}

\theoremstyle{definition}
 \mynewtheorem{definition}{Definition}

\theoremstyle{plain}
 \mynewtheorem{proposition}{Proposition}
 \mynewtheorem{lemma}{Lemma}
 \mynewtheorem{theorem}{Theorem}
 \mynewtheorem{corollary}{Corollary}
 
\theoremstyle{remark}
 \mynewtheorem{remark}{Remark}

\begin{document}
 \title{A Generalization of Principal Bundles With a Parabolic or Level Structure}

\author{Nikolai Beck}
\address{%
 Mathematisches Institut, BTU Cottbus--Senftenberg, PF 101344, 03013 Cottbus, Germany}
\email[N.~Beck]{nikolai.beck@b-tu.de}


\begin{abstract}
  We define a parameter dependent notion of stability for principal bundles with a certain local decoration, which generalizes both parabolic and level structures, and construct their coarse moduli space. A necessary technical step is the construction of the moduli space of tuples of vector bundles with a global and a local decoration, which we call decorated tumps. We introduce a notion of asymptotic stability for decorated tumps and show, that stable decorated principal bundles can be described as asymptotically stable decorated tumps.
\end{abstract}

 \maketitle
 
\section*{Introduction}
\subsection*{Motivation}
Let $X$ be a smooth projective curve over $\CC$ and $x_0$ a closed point in $X$. There are different examples of vector bundles with a decoration over $x_0$: A parabolic structure on a vector bundle $E$ is a weighted flag in the fiber over $x_0$. By a generalization of the Narasimhan--Seshadri theorem there is a bijection between stable parabolic vector bundles on $X$ and irreducible unitary representations of the fundamental group $\pi_1(X)$ \cite{simpson1990}. A level structure of $E$ is a completed homomorphism $E_{|x_0}\Rightarrow \CC^{\rk(E)}$. These objects play an important role in the compactification of the stack of shtukas using GIT (see \cite{Ngo2007}). In our previous paper \cite{Beck2014DecSwamps} we constructed the moduli space of decorated vector bundles, which generalize both examples. The aim of this article is to establish similar results for decorated principal bundles.

Let $G$ be an affine reductive group and $\sigma$ a representation. We call a pair of a principal $G$-bundle $\PB$ and a point $s$ in the associated bundle $\PB_\sigma$ over $x_0$ a \emph{decorated principal bundle}. We define a notion of stability for decorated principal bundles, depending on a character of $G$ and positive rational number. The main result is the existence of the (projective) moduli space of (semi-)stable decorated principal bundles (see Theorem \ref{thm:moduli_space_of_dec_PB} for a precise statement).

If $Q\subset G$ is a parabolic subgroup and $\sigma$ is the natural action of $G$ on the generalized flag variety $G/Q$, then a decorated principal bundle is a parabolic principal bundle. Our construction thus generalizes the result of Heinloth and Schmitt \cite{heinloth-schmitt}, who constructed the moduli space of parabolic principal bundles with semisimple structure group.

If $\sigma$ is the action of $G$ on its wonderful compactification $\bar{G}$, then a decorated principal bundle is a principal bundle with a level structure. In analogy to the vector bundle case these should be useful for the compactification of the stack of $G$-shtukas via GIT.

\subsection*{Outline}
The construction of the moduli space is carried out following the strategy for principal bundles explained in \cite{Schmitt08}. In Section \ref{sec:prelims} we introduce the necessary notation for tuples of vector bundles, before we construct the moduli space of decorated tumps in Section \ref{sec:Moduli_of_tumps}. The construction is very similar to the construction of moduli space of decorated swamps from \cite{Beck2014DecSwamps}, and we  focus on the differences. Additionally, we introduce the notion of asymptotic stability and show, that for a large enough parameter the notion of stability is equivalent to that of asymptotic stability. This generalizes results from \cite{Beck2014AsymptoticStability}.

 In Section \ref{sec:Moduli_of_PB} we fix a closed embedding of $G$ into a product of general linear groups. This allows to describe a decorated principal bundle as a tuple of vector bundles with a reduction of the structure group and a local decoration. For technical reasons we need to encode this reduction as a section in an associated projective bundle constructed with a homogeneous representation. This leads to the notion of decorated  tumps. We show, that stability of a decorated principal bundle is equivalent to asymptotic stability of the corresponding decorated tump. The moduli space is then constructed as a GIT-quotient of a certain parameter space.
 
 Finally, Section \ref{sec:examples} contains the examples of decorated principal bundles, namely parabolic principal bundles and principal bundles with a level structure.

 \subsection*{Notation and Conventions}
 In this work we will identify a geometric vector bundle $E$ with its sheaf of sections. If $F$ is a subsheaf of $E$, the subbundle generically generated by $F$ is
\[
 \ker(E \to (E/F)/T) \,,
\]
where $T$ is the torsion subsheaf of $E/F$. We denote by $\PP(E)$ the hyperplane bundle $\opname{Proj}(\opname{Sym}^*E)$. For $x\in \RR$ we set $[x]_+:=\max\{x,0\}$. If $X$ and $Y$ are schemes, we denote by $\pr_X$ and $\pr_Y$ the canonical projection from the product $X\times Y$ to $X$ and $Y$ respectively.

\subsection*{Acknowledgement}
This article presents the main result of the author's PhD thesis \cite{Beck2014}. The author would like thank his advisor A. Schmitt for his guidance.

\section{Preliminaries}
\label{sec:prelims}
In order to fix the notation we recall some facts about split vector bundles.
\subsection{Split Vector Spaces}
 Let $T$ be a finite index set.
\begin{definition}
   A \emph{$T$-split vector space} is $T$-tuple of complex vector spaces $V=(V^t,t\in T)$. A \emph{homomorphism} $f:V\to W$ of $T$-split vector spaces is a tuple of homomorphisms $(f^t:V^t\to W^t, t\in T)$. 
\end{definition}
The category of $T$-split vector spaces is an abelian category. The \emph{dimension vector} of a $T$-split vector space $V$ is the tuple $\tup{d}:=(\dim(V^t),t\in T)$. Its automorphism group is 
\[
 \GL_T(V):=\prod_{t\in T} \GL(V^t)\cong \GL(\tup{d},\CC):= \prod_{t\in T}\GL(d_t,\CC)\,.
\]
For positive rational numbers $\ka_t$, $t\in T$, the \emph{$\tup{\ka}$-dimension} of $V$ is  $\dim_{\tup{\ka}}(V):=\sum_{t\in T}\ka_t \dim(V^t)$. We also define
\[
 \SL_{T}^{\tup{\ka}}(V):=\left\{(g_t,t\in T)\in \GL_T(Y)\,\bigg|\, \prod_{t\in T}\det(g_t)^{\ka_t}=1\right\}\,.
\]

\begin{remark}
Suppose $\tup{\ka}$ is a $T$-tuple of positive integers. We call $V^{\oplus\tup{\ka}}:=\bigoplus_{t\in T}{V^t}^{\oplus\ka_t}$ the \emph{$\tup{\ka}$-total space} of $V$. There is a natural embedding
\[
 \iota_{\tup{\ka}}:\GL_T(V)\to \GL(V^{\oplus\tup{\ka}})\,.
\]
Note that we have $\dim_{\tup{\ka}}(V)=\dim(V^{\oplus\tup{\ka}})$ and $\SL_T^{\tup{\ka}}(V)=\iota_{\tup{\ka}}^{-1}(\SL(V^{\oplus\tup{\ka}}))$.
\end{remark}

\begin{definition}
 A \emph{weighted flag} of length $k$ of a $T$-split vector space $V$ is a pair $(V_\bullet,\tup{\alpha})$, where $V_\bullet$ is a flag $ 0= V_0\subset\ldots \subset V_{k+1}=V$ of $T$-split subspaces and $\tup{\alpha}$ is a $k$-tuple of positive rational numbers.
\end{definition}

Given a one-parameter subgroup $\la$ of $\SL_T^{\tup{\ka}}(V)$, there are a number $k$, weights $\ga_i< \ldots< \ga_{k+1}$ and a decomposition of $V$ into $T$-split eigenspaces $V^i$, $1\le i \le k+1$, such that $\la(c)\cdot v=c^{\ga_i}v$ for all $t\in T$, $v\in (V^i)^t$ and $c\in \CC^*$. We define the associated weighted flag $(V_\bullet,\tup{\alpha})$ of $\la$ by setting $V_j:=\bigoplus_{i=1}^j V^i$ and $\alpha_j:=(\ga_{j+1}-\ga_{j})/\dim_{\tup{\ka}}(V)$ for $j=1,\ldots,k$.

Conversely, if $(V_\bullet,\tup{\alpha})$ is a weighted flag, we choose a decomposition $V=\bigoplus_{i=1}^{k+1} V^i$, such that $V_j=\bigoplus_{i=1}^j V^i$. Furthermore we define
\[
 \ga_i:= \sum_{j=1}^k \alpha_j\dim_{\tup{\ka}}(V_j)-\sum_{j=i}^k \alpha_j\dim_{\tup{\ka}}(V)\,,\qquad i=1,\ldots,k+1\,.
\]
Let $m$ be the least common denominator of $\ga_1,\ldots,\ga_{k+1}$ and define a one-parameter subgroup $\la$ of $\SL_T^{\tup{\ka}}(V)$ by $\la(c)\cdot v:= c^{m\ga_i}v$ for $v\in V^i$, $1\le i \le k+1$ and $c\in\CC^*$. Then $(V_\bullet,m\tup{\alpha})$ is the associated weighted flag of $\la$

\begin{definition}
A representation $\rho:\GL(\tup{d},\CC)\to \GL(V)$ is \emph{homogeneous} of degree $\ga$, if $\rho(c\cdot \id)=c^{\ga}\id_V$ for all $c\in \CC^*$.
\end{definition}
\begin{proposition} \label{prop:W_abc}
 Let $W$ be a $T$-split vector space, $\tup{\ka}\in\ZZ^T_{> 0}$ and $\rho:\GL_T(W)\to \GL(V)$ a homogeneous representation of degree $\ga$. Then there are natural numbers $a,b,c$ with $\ga=a-c\dim_{\tup{\ka}}(W)$, such that $V$ is a direct summand of the representation
 \[
  W^{\oplus\tup{\ka}}_{a,b,c}:=\left({W^{\oplus\tup{\ka}}}^{\otimes a}\right) ^{\oplus b}\otimes \left( \bigwedge^{\dim_{\tup{\ka}}(W)} W^{\oplus\tup{\ka}}\right)^{\otimes -c}\,.
 \]
\end{proposition}
\begin{proof}
 This is Proposition 2.5.1.2 in \cite{Schmitt08}.
\end{proof}

\subsection{Split Vector Bundles}
 Let $X$ be a smooth projective curve over the complex numbers and $T$ a finite index set.
\begin{definition}
 A \emph{$T$-split sheaf} on $X$ is a tuple of coherent sheaves $\mc{F}=(\mc{F}^t,t\in T)$.  A \emph{morphism} $f:\mc{F}\to \mc{G}$ of $T$-split sheaves is tuple of morphisms $(f^t:\mc{F}^t\to \mc{G}^t,t\in T)$. 
\end{definition}
The \emph{rank vector} of a $T$-split sheaf $\mc{F}$ is $\tup{r}(\mc{F})=(\rk(\mc{F}^t),t\in T)$, its \emph{degree vector} is $\tup{d}(\mc{F})=(\deg(\mc{F}^t),t\in T)$.
\begin{remark}
 The category of $T$-split sheaves is an abelian category.
 \end{remark}
 
\begin{definition}
Let $\mc{F}$ be a $T$-split sheaf, $\tup{\ka}\in \QQ_{>0}^T$ and $\tup{\chi}\in \QQ^T$. The \emph{$\tup{\chi}$-rank} and the \emph{$(\tup{\ka},\tup{\chi})$-degree} of $\mc{F} $ are 
\begin{align*}
  \rk_{\tup{\chi}}(\mc{F}):=\sum_{t\in T} \chi_t \rk(\mc{F}^t)\,,
 && \deg_{\tup{\ka},\tup{\chi}}(\mc{F}):=\sum_{t\in T} \left( \ka_t \deg(\mc{F}^t) + \chi_t \rk(\mc{F}^t)\right)\,,
 \end{align*}
respectively. If $\rk_{\tup{\chi}}(\mc{F})\neq 0$, then the \emph{$(\tup{\ka},\tup{\chi})$-slope} of $\mc{F}$ is
\[
 \mu_{\tup{\ka},\tup{\chi}}(\mc{F}):= \frac{\deg_{\tup{\ka},\tup{\chi}}(\mc{F})}{\rk_{\tup{\ka}}(\mc{F})}\,.
\]
\end{definition}
\begin{remark}
Suppose $\tup{\ka}$ is $T$-tuple of positive integers. Then, the rank of the \emph{$\tup{\ka}$-total sheaf} $\mc{F}^{\oplus\tup{\ka}}:=\bigoplus_{t\in T}\mc{F}^{t\, \oplus\ka_t}$ is given by $\rk(\mc{F}^{\oplus\tup{\ka}})=\rk_{\tup{\ka}}(\mc{F})$.
\end{remark}
\begin{definition} \label{def:split_VB}
 A \emph{$T$-split vector bundle} is a $T$-split sheaf $E=(E^t,t\in T)$, such that for every $t\in T$ the sheaf $E^t$ is a vector bundle. A \emph{morphism} of $T$-split vector bundles is a morphism of $T$-split sheaves. 
\end{definition}
\begin{remark}
 The datum of a $T$-split vector bundle with rank vector $\tup{r}$ is equivalent to the datum of principal bundle with structure group $\GL(\tup{r},\CC)$.
\end{remark}

\begin{definition}
 A $T$-split vector bundle $E$ is $(\tup{\ka},\tup{\chi})$-\emph{\sstable{}} if every $T$-split proper subsheaf $\mc{F}$ with $\rk_{\tup{\ka}}(\mc{F})\neq 0$ satisfies the condition
 \[
  \mu_{\tup{\ka},\tup{\chi}}(\mc{F}) (\le) \mu_{\tup{\ka},\tup{\chi}}(E)\,.
 \]
\end{definition}
\begin{remark}
 \begin{enumerate}[wide]
  \item If $|T|=1$, then a $T$-split vector bundle $E$ is nothing but an ordinary vector bundle, and $E$ is $(\tup{\ka},\tup{\chi})$-\sstable{} if and only if $E$ is \sstable{} as a vector bundle.
  \item As in the case of vector bundles one can show that a $T$-split vector bundle is $(\tup{\ka},\tup{\chi})$-\sstable{} if and only if $\mu_{\tup{\ka},\tup{\chi}}(F) (\le) \mu_{\tup{\ka},\tup{\chi}}(E)$ holds for every non-trivial proper $T$-split subbundle $F\subset E$.
  \item Because the $\tup{\ka}$-rank and the $(\tup{\ka},\tup{\chi})$-degree are additive for short exact sequences, we get
  \[
   \mu_{\tup{\ka},\tup{\chi}}(F) < \mu_{\tup{\ka},\tup{\chi}}(E)+C \qquad\Longleftrightarrow \qquad
   \mu_{\tup{\ka},\tup{\chi}}(E) < \mu_{\tup{\ka},\tup{\chi}}(E/F)+C\frac{\rk_{\tup{\ka}}(F)}{\rk_{\tup{\ka}}(E/F)}\
  \]
for all $T$-split subbundles $F\subset E$ with $\rk_{\tup{\ka}}(F)\neq 0\neq \rk_{\tup{\ka}}(E/F)$.
\item Because the category of $T$-split sheaves is abelian, there exists a unique Harder--Narasimhan-filtration of a $T$-split vector bundle $E$. This allows us to define the maximal and the minimal slope $\mu_{\max}(E):=\mu_{\max,\tup{\ka},\tup{\chi}}(E)$ and $\mu_{\min}(E):=\mu_{\min,\tup{\ka},\tup{\chi}}(E)$ respectively.
 \end{enumerate}
\end{remark}
\begin{definition}
 A \emph{weighted flag} of a $T$-split vector bundle $E$ is a flag of $T$-split subbundles $E_\bullet$ of length $l=\len(E_\bullet)$ together with a \emph{weight vector} $\tup{\alpha}\in\QQ_{>0}^l$.
\end{definition}
For a weighted flag $(E_\bullet,\tup{\alpha})$ we set
\[
 M_{\tup{\ka},\tup{\chi}}(E_\bullet,\tup{\alpha}):=\sum_{j=1}^{\len(E_\bullet)}\alpha_j\left(\deg_{\tup{\ka},\tup{\chi}}(E)\rk_{\tup{\ka}}(E_j) - \deg_{\tup{\ka},\tup{\chi}}(E_j)\rk_{\tup{\ka}}(E) \right)\,.
\]
It is obvious, that a $T$-split vector bundle $E$ is \sstable{} if and only if the condition
\[
  M_{\tup{\ka},\tup{\chi}}(E_\bullet,\tup{\alpha})(\ge) 0
\]
holds for every weighted flag $(E_\bullet,\tup{\alpha})$ of $E$.

\begin{lemma} \label{lem:Split_bundle_semistable}
Let $E\neq 0$ be a $T$-split vector bundle and set $T(E):=\{t\in T\,|\, E^t \neq 0\}\subset T$. 
 \begin{enumerate}
  \item If $E$ is $(\tup{\ka},\tup{\chi})$-stable, then $|T(E)|=1$.
  \item If $E$ is $(\tup{\ka},\tup{\chi})$-semistable, then for every index $t\in T(E)$ the vector bundle $E^t$ is semistable of slope $\mu(E^t)=\mu_{\tup{\ka},\tup{\chi}}(E)-\chi_t/\ka_t$.
 \end{enumerate}
\end{lemma}
\begin{proof}
Let $t_0\in T(E)$ and consider the $T$-split vector bundle $F=(F^t,t\in T)$ with $F^{t_0}:=E^{t_0}$ and $F^t:=0$ for $t\neq t_0$.

 (i)  If $|T(E)|>1$, then $F$ is both a proper subbundle and a proper quotient of $E$, so $E$ cannot be $(\tup{\ka},\tup{\chi})$-stable.
 
 (ii) If $E$ is $(\tup{\ka},\tup{\chi})$-semistable, we find $\mu_{\tup{\ka},\tup{\chi}}(E)=\mu_{\tup{\ka},\tup{\chi}}(F)= \mu(E^{t_0})+\frac{\chi_{t_0}}{\ka_{t_0}}$.
\end{proof}

Let us abbreviate
\begin{align*}
  m(\tup{\ka},\tup{\chi}):=\min\left\{\chi_t/\ka_t\,|\, t\in T\right\} \,, && M(\tup{\ka},\tup{\chi}):=\max\left\{\chi_t/\ka_t\,|\, t\in T\right\}\,.
\end{align*}
For a $T$-split vector bundle $E$ we define
\[
 h^0_{\tup{\ka},\tup{\chi}}(E):=\sum_{t\in T}\left( \ka_th^0(E^t)+\chi_t\rk(E^t)\right)
\]
and $h^0_{\tup{\ka}}(E):=h^0_{\tup{\ka},\tup{0}}(E)$. Furthermore, we set $[x]_+:=\max\{x,0\}$. As a generalization of the Le Potier--Simpson estimate (\cite{LePot97}, Lemma 7.1.2) we find:
\begin{proposition} \label{prop:LS-estimate}
 Let $E$ be a $T$-split vector bundle.
 \begin{enumerate}
  \item If $E$ is semistable, we have
  \[
   h^0_{\tup{\ka}}(E)\le \rk_{\tup{\ka}}(E)\left[\mu_{\tup{\ka},\tup{\chi}}(E)-m(\tup{\ka},\tup{\chi})+1\right]_+\,.
  \]
\item An arbitrary $T$-split vector bundle $E$ of $\tup{\ka}$-rank $r$ satisfies
\[
  h^0_{\tup{\ka}}(E)\le (r-1)\left[\mu_{\max}(E)-m(\tup{\ka},\tup{\chi})+1\right]_+ + \left[\mu_{\min}(E)-m(\tup{\ka},\tup{\chi}) +1\right]_+\,.
\]
 \end{enumerate}
\end{proposition}
\begin{proof}
 (i) By Lemma \ref{lem:Split_bundle_semistable}, (ii), and the Le Potier--Simpson estimate for semistable vector bundles we find $h^0(E^t)\le \rk(E^t)[\mu(E^t)+1]_+$ for every $t\in T$. Then, the the explicit formula for the slope of $E^t$ the result.
 
 (ii) Let $E_\bullet$ be the Harder--Narasimhan filtration of $E$. From part (i) we get
   \begin{align*}
    h^0_{\tup{\ka}}(E) \le & \sum_{i=1}^n h^0_{\tup{\ka}}(E_i/E_{i-1})\\
     \le  & \sum_{i=1}^n \rk_{\tup{\ka}}(E_i/E_{i-1})\left[\mu_{\tup{\ka},\tup{\chi}}(E_i/E_{i-1})-m(\tup{\ka},\tup{\chi})+1\right]_+\,.
   \end{align*}
   Since $\mu_{\tup{\ka},\tup{\chi}}(E_i/E_{i-1})\le\mu_{\max}(E)$, $i=1,\ldots,n-1$, and $\rk_{\tup{\ka}}(E/E_{n-1})\ge 1$, the claim follows.
\end{proof}

\section{Moduli of Decorated Tuples of Vector Bundles}
\label{sec:Moduli_of_tumps}
In this section we define (semi-)stable decorated tumps. The main results are the existence of their moduli space and the equivalence of asymptotic stability with stability for a large enough parameter. Since the constructions and proofs are similar to those for the case of decorated swamps in \cite{Beck2014DecSwamps,Beck2014AsymptoticStability}, we will only give an outline of this construction.
\subsection{Stable Decorated Tumps}
Let $\tup{r}$ be a tuple of positive integers and $\rho:\GL(\tup{r},\CC)\to \GL(V)$ a representation. A $T$-split vector bundle $E$ of rank vector $\tup{r}$ corresponds to $\GL(\tup{r},\CC)$-bundle $\PB$. We denote by $E_\rho$ the associated vector bundle $\PB\times^\rho V$. 
\begin{definition}
 Let $\rho:\GL(\tup{r},\CC)\to \GL(V_1)$ and $\sigma:\GL(\tup{r},\CC)\to \GL(V_2)$ be two homogeneous representations. A \emph{$\sigma$-decorated $\rho$-tump} of type $(\tup{d},l)$ is a tuple $(E,L,\phi,s)$, where $E$ is a split vector bundle with $\rk(E^t)=r_t$ and $\deg(E^t)=d_t$, $t\in T$, $L$ is a line bundle of degree $l$ on $X$, $\phi:E_\rho\to L$ is a non-trivial homomorphism and $s\in E_{\sigma|\{x_0\}}^\vee$ is a point.
 
 Two $\sigma$-decorated $\rho$-tumps $(E,L,\phi,s)$ and $(E',L',\phi',s')$ are isomorphic if there are isomorphisms $f:E\to E'$, $h:L\to L'$ and a number $c\in \CC^*$, such that $\phi'\circ f_\rho=h\circ \phi$ and $s'\circ f_{\sigma|\{x_0\}}=cs$. Here, $f_\rho:E_\rho\to E'_\rho$ and $f_\sigma:E_\sigma\to E'_\sigma$ are the  isomorphisms induced by $f$.
\end{definition}
In the following we will fix two homogeneous representations $\rho$ and $\sigma$ as well as a degree vector $\tup{d}$ and a number $l$, and refer to $\sigma$-decorated $\rho$-tumps of type $(\tup{d},l)$ simply as decorated tumps.

A weighted flag $(E_\bullet,\tup{\alpha})$ of a $T$-split vector bundle $E$ induces weighted flags $(E_{\rho,\bullet},\tup{\alpha}_\rho)$ and $(E_{\sigma,\bullet},\tup{\alpha}_\sigma)$ of $E_\rho$ and $E_\sigma$ respectively (see Definition 2.8 in \cite{Beck2014DecSwamps}). Restricting these to the generic point $\eta\in X$ and to $x_0$ we obtain flags $\EE_\bullet:=E_{\rho,\bullet|\{\eta\}}$ and $F_\bullet:=E_{\sigma,\bullet|\{x_0\}}$ respectively. We set
\[
 \mu_1(E_\bullet,\tup{\alpha},\phi):=\mu(E_{\bullet,\rho},\tup{\alpha}_\rho,[\phi])\,,\qquad 
 \mu_2(E_\bullet,\tup{\alpha},s):=\mu(E_{\bullet,\sigma},\tup{\alpha}_\sigma,[s])\,,
\]
where $[\phi]\in \PP(E_{\rho|\eta})$ and $[s]\in\PP(E_{\sigma|\{x_0\}})$ denote the points induced by $\phi$ and $\sigma$, respectively.
\begin{definition}
 We call a $\sigma$-decorated $\rho$-tump \emph{$(\delta_1,\delta_2,\tup{\ka},\tup{\chi})$-\sstable{}} if the condition
 \[
  M_{\tup{\ka},\tup{\chi}}(E_\bullet,\tup{\alpha})+\delta_1\mu_1(E_\bullet,\tup{\alpha},\phi)+\delta_2\mu_2(E_\bullet,\tup{\alpha},s) (\ge) 0
 \]
holds for every weighted flag $(E_\bullet,\tup{\alpha})$ of $E$.
\end{definition}
\begin{remark}
\begin{enumerate}[wide]
\item The stability condition remains unchanged if one substitutes $\tup{\ka}$, $\tup{\chi}$, $\delta_1$ and $\delta_2$ by $\tup{\ka}':=c\tup{\ka}$, $\tup{\chi}':=c\tup{\chi}$, $\delta'_1:=c^2\delta_1$ and $\delta'_2:=c^2\delta_2$ for some $c\in \QQ_{>0}$. We may thus assume that $\tup{\ka}$ is integral.
\item Furthermore, $M_{\tup{\ka},\tup{\chi}}$ is unchanged under the replacement $\tup{\chi}\mapsto \tup{\chi}+ c\tup{\ka}$ for $c\in \QQ$. One can choose, e.g., $\rk_{\tup{\chi}}(E)=0$ or $\deg_{\tup{\ka},\tup{\chi}}(E)=0$. We will not use this in the following. 
\end{enumerate}
\end{remark}

\begin{remark} \label{rem:tmp_stab}
Let $W=(\CC^{r_t},t\in T)$. Since $\rho$ and $\sigma$ are homogeneous, there are numbers
 $a_m,b_m,c_m\in \ZZ_{\ge0}$, such that $V_m$ is a direct summand
of $W^{\oplus\tup{\ka}}_{a_m,b_m,c_m}$, $m=1,2$. Thus, there are surjections
$p_1\colon E^{\oplus\tup{\ka}}_{a_1,b_1,c_1}\to E_\rho$ and $p_2\colon E^{\oplus\tup{\ka}}_{a_2,b_2,c_2}\to E_\sigma$. For a flag $E_\bullet$ of $E$, $m=1,2$ and a tuple $\tup{i}\in
I_m:=\{1,\ldots,l\}^{a_m}$ we set
 \[
  E^{\otimes\tup{i}}:=(E_{i_1}^{\oplus\tup{\ka}}\otimes \cdots \otimes
E_{i_{a_m}}^{\oplus\tup{\ka}})^{\oplus b_m}\otimes \left(\bigwedge^r
E^{\oplus\tup{\ka}}\right)^{\otimes -c_m} \subset E^{\oplus\tup{\ka}}_{a_m,b_m,c_m}\,
 \]
with $r:=\rk_{\tup{\ka}}(E)$. Then, one finds
 \begin{align} \label{eq:tmp_stab1}
  \mu_1(E_\bullet,\tup{\alpha},\phi)&=-\min\left\{\sum_{j=1}^{\len(\tup{\alpha})}
\alpha_j(a_1\rk_{\tup{\ka}}(E_j)-r \nu_j(\tup{i}))
\,\biggm|\,(\phi\circ p_1)_{|E^{\otimes\tup{i}}} \neq 0 \,,\tup{i}\in I_1 \right\},\\
  \label{eq:tmp_stab2}
  \mu_2(E_\bullet,\tup{\alpha},s)&=-\min\left\{ \sum_{j=1}^{\len(\tup{\alpha})}
\alpha_j(a_2\rk_{\tup{\ka}}(E_j)-r \nu_j(\tup{i}))  \,\biggm|\,
(s\circ p_2)_{|E_{|x_0}^{\otimes\tup{i}}}\neq 0\,, \tup{i}\in I_2\right\}
 \end{align}
with $\nu_j(\tup{i}):=\#\{i\in \tup{i}\,|\,i\le j \}$.
\end{remark}

\subsection{S-Equivalence}
 Let $(E,L,\phi,s)$ be a $(\delta_1,\delta_2,\tup{\ka},\tup{\chi})$-semistable decorated tump.
The associated graded bundle of a flag $E_\bullet$ of $E$ is the $T$-split vector bundle
\[
 E_{\gr}:=\bigoplus_{i=1}^{\len(E_\bullet)}E_i/E_{i-1}=\left(\bigoplus_{i=1}^{\len(E_\bullet)}E^t_i/E^t_{i-1},t\in T\right)\,.
\]
The flag $E_\bullet$ also induces flags $E_{\rho,\bullet}$ of $E_\rho$ and $E_{\sigma,\bullet}$ of $E_\sigma$. One checks that $(E_{\rho})_{\gr}=(E_{\gr})_{\rho}$ and $(E_{\sigma})_{\gr}=(E_{\gr})_{\sigma}$ (see \cite{Beck2014DecSwamps}, \S 3.3). Set $i_0:=\min\{i=1,\ldots,\len(E_{\rho,\bullet})\,|\, \phi_{|E_{\rho,i}}\neq 0\}$ and $j_0:=\min\{j=1,\ldots,\len(E_{\sigma,\bullet})\,|\, s_{|E_{\sigma,j|\{x_0\}}}\neq 0 \}$. Then, the restrictions $\phi_{|E_{\rho,i_0}}$ and $s_{|E_{\sigma,j_0}}$ define non-trivial homomorphisms
\[
 \phi_0:E_{\rho,i_0}/E_{\rho,i_0-1}\to L\, \qquad s_0:E_{\sigma,j_0}/E_{\sigma.j_0-1}\to \CC\,.
\]
Let $\phi_\gr:E_{\rho,\gr}\to L$ and $s_\gr:E_{\sigma,\gr}\to \CC$ denote the composition of these homomorphisms with the natural projections. We call
\[
 \df_{(E_\bullet,\tup{\alpha})}(E,L,\phi,s):=(E_\gr,L,\phi_\gr,s_\gr)
\]
the \emph{admissible deformation} of $(E,L,\psi,s)$ along the flag $(E_\bullet,\tup{\alpha})$.
\begin{definition}
 We call a weighted flag $(E_\bullet,\tup{\alpha})$ of $E$ \emph{critical} with respect to $(E,L,\phi,s)$ if 
 \[ 
 M_{\tup{\ka},\tup{\chi}}(E_\bullet,\tup{\alpha})+\delta_1\mu_1(E_\bullet,\tup{\alpha},\phi)+\delta_2\mu_2(E_\bullet,\tup{\alpha},s) =0\,.
 \]
\end{definition}

\begin{definition}
 We define \emph{S-equivalence} as the equivalence relation generated by isomorphisms and 
 \[
  (E,L,\phi,s) \sim_S \df_{(E_\bullet,\tup{\alpha})}(E,L,\phi,s)\,.
 \]
for all critical weighted flags $(E_\bullet,\tup{\alpha})$ of $E$.
\end{definition}

\subsection{Families of Decorated Tumps}
Let $\Jac^l$ denote the Jacobian of line bundles of degree $l$ on $X$ and fix a Poincar\'e bundle $\mc{L}$ on $\Jac^l\times X$.
\begin{definition} \label{def:family_of_dec_tumps}
 A \emph{family of decorated tumps parametrized by a scheme $S$} is tuple 
 \[
 \mc{F}=(E_S,\ka_S,N_{1,S},N_{2,S},\phi_S,s_S)
 \]
 where
 \begin{itemize}
  \item $E_S$ is a $T$-split vector bundle on $S\times X$ with rank vector $\tup{r}(E_S)=\tup{r}$, such that for every point $y\in S$, the $T$-split vector bundle $E_{S|\{y\}\times X}$ has degree vector $\tup{d}$,
  \item $\ka_S:S\to \Jac^l$ is a morphism,
  \item $N_{1,S}$, $N_{2,S}$ are two line bundles on $X$,
  \item $\phi_S:E_{S,\rho}\to \pr_S^*N_{1,S}\otimes (\ka_S\times\id_X)^*\mc{L}$ is a homomorphism such that for each point $y\in S$ the restriction $\phi_{S|\{y\}\times X}$ is non-trivial,
  \item and $s:E_{S,\sigma}\to N_{2,S}$ is a surjective homomorphism.
 \end{itemize}

 Two families $\mc{F}$ and $\mc{F}'$ parametrized by $S$ are \emph{isomorphic} if
 \begin{itemize}
  \item there are a line bundle $L_S$ on $S$ and an isomorphism $f:E_S\to E'_S\otimes \pr_S^*L_S$,
  \item $\ka_S=\ka'_S$ holds,
  \item there is an isomorphism $h_1:N_{1,S}\to N'_{1,S}\otimes L_S^{\otimes \deg(\rho)}$ such that
  \[
    (\phi'_S\otimes\id_{\pr_S^*L_S}) \circ f_\rho=(\pr_S^*h_1\otimes\id_{(\ka_S\times\id_X)^*\mc{L}} )\circ \phi_S\,,
  \]
  \item and there is an isomorphism $h_2:N_{2,S}\to N'_{2,S}\otimes L_S^{\otimes\deg(\sigma)}$ such that
  \[
  (s'_S\otimes \id_{L_S^{\otimes \deg(\sigma)}}) \circ f_{\sigma|S\times \{x_0\}}= h_2\circ s_S\,.
  \]
 \end{itemize}
\end{definition}

\begin{definition}
 We call a family of decorated tumps $\mc{F}$ parametrized by a scheme $S$ \emph{$(\delta_1,\delta_2,\tup{\ka},\tup{\chi})$-\sstable{}} if for every point $y\in S$ the decorated tump $\mc{F}_{|\{y\}}$ is  $(\delta_1,\delta_2,\tup{\ka},\tup{\chi})$-\sstable.
\end{definition}
These notions determine the moduli functor of (S-equivalence classes of) $(\delta_1,\delta_2,\tup{\ka},\tup{\chi})$-\sstable{} decorated tumps. Recall that in Remark \ref{rem:tmp_stab} we have fixed numbers $a_m,b_m,c_m$, $m=1,2$, such that $\rho$ and $\sigma$ are direct summands of the representation on $W^{\oplus\tup{\ka}}_{a_m,b_m,c_m}$. The main result of this section is the following:
\begin{theorem} \label{thm:moduli_space_of_tumps}
For $a_2\delta_2<1$, the \textup{(}projective\textup{)} moduli space of $(\delta_1,\delta_2,\tup{\ka},\tup{\chi})$-\sstable{} decorated tumps exists. 
\end{theorem}
The construction of this moduli space will be carried out in Sections \ref{subsec:decorated_quotient_tumps} trough \ref{subsec:proof_moduli_of_tumps}.

\subsection{Decorated Quotient Tumps}
\label{subsec:decorated_quotient_tumps}
The parameter space of decorated tumps is constructed as the fine moduli space of another moduli problem. We start with a result on boundedness.
\begin{lemma}
 There is constant $C$, such that for every $(\delta_1,\delta_2,\tup{\ka},\tup{\chi})$-semistable decorated tump $(E,L,\phi,s)$ of type $(\tup{d},l)$ we have:
 \begin{enumerate}
  \item Every non-trivial $T$-split subbundle $F$ satisfies $\mu_{\tup{\ka},\tup{\chi}}(F)< \mu_{\tup{\ka},\tup{\chi}}(E) + C$.
  \item For all $t\in T$ every non-trivial proper subbundle $F^t\subset E^t$ satisfies 
  \[
  \mu(F^t)< \mu(E)+C-m(\tup{\ka},\tup{\chi})\,.
  \]
 \end{enumerate}
\end{lemma}
\begin{proof}
 Consider the weighted flag $(E_\bullet,\tup{\alpha}):=(F\subset E,(1))$. From \eqref{eq:tmp_stab1} and \eqref{eq:tmp_stab2} we obtain the inequalities
 \[
  \mu_1(E_\bullet,\tup{\alpha},\phi)\le a_1(r-1)\,, \qquad \mu_2(E_\bullet,\tup{\alpha},s)\le a_2(r-1)\,.
 \]
 The claim now follows from the $(\delta_1,\delta_2,\tup{\ka},\tup{\chi})$-semistability with respect to $(E_\bullet,\tup{\alpha})$.
\end{proof}
The usual arguments show, that there is an $n_0$ such that for every $n\ge n_0$, every $(\delta_1,\delta_2,\tup{\ka},\tup{\chi})$-\sstable{} decorated tump $(E,L,\phi,s)$ and every index $t\in T$ the bundle $E^t(n)$ is globally generated and $H^1(E^t(n))$ vanishes. We choose such an $n$ and fix complex vector spaces $Y^t$ of dimension $p_t(n):=d_t+ r_t(n+1-g)$, $t\in T$.

\begin{definition}
 A \emph{family of decorated quotient tumps} parametrized by a scheme $S$ is a tuple 
 \[
 (q_S,\ka_S,N_{1,S},N_{2,S},\phi_S,s_S)
 \]
 where $q_S:Y\otimes\pr_X^*\mc{O}_X(-n)\to E_S$ is a quotient of $T$-split vector bundles on $S\times X$, such that $(E_S,\ka_S,N_{1,S},N_{2,S},\phi_S,s_S)$ is a family of decorated tumps parametrized by $S$ and 
 \[
  \pr_{S*}(q_S\otimes\id_{\pr_X^*\mc{O}_X(n)}):Y\otimes \mc{O}_S \to \pr_{S*}(E_S\otimes \pr_X^*\mc{O}_X(n)) 
 \]
is an isomorphism.
\end{definition}

There is a natural notion of isomorphisms similar to Defintion \ref{def:family_of_dec_tumps}. This determines the moduli functor of decorated quotinet tumps.

\begin{proposition} \label{prop:QTmp_exists}
 The fine moduli space $\QTmp_n$ of decorated quotient tumps of type $(\tup{d},l)$ exists.
\end{proposition}
\begin{proof}
For $t\in T$ let $\Quot^t_n$ be the Quot scheme parameterizing quotients of rank $r_t$ and degree $d_t$ of $Y^t \otimes\mc{O}_X(-n)$ on $X$, and let $q^t:Y^t\otimes\pr_X^*\mc{O}_X(-n)\to Q^t$ be the universal quotient on $\Quot^t_n\times X$. Further, Let $\Quot^{t,0}_n$ be the open subscheme, such that for every point $p\in \Quot^{t,0}_n$ the sheaf $Q^t_p$ is locally free and $H^0(q^t_p(n))$ is an isomorphism. On $\Quot_n:=\prod_{t\in T}\Quot_n^{t,0}$ we have the quotient 
\[
q:Y\otimes\pr_X^*\mc{O}_X(-n)\to Q:=(\pr_{\Quot_n^{t,0}}^*Q^t,t\in T)
\]
of $T$-split vector bundles. The construction of $\QTmp_n$ now proceeds as in the proof of Proposition 4.3 in \cite{Beck2014DecSwamps} with vector bundles replaced by $T$-split vector bundles.
\end{proof}

\subsection{The Gieseker Space}
We construct an equivariant injective morphism from the paramter space into a projective space:
For $t\in T$ let $\Jac^{d_t}$ be the Jacobian of line bundles of degree $d_t$ on $X$ and let $\mc{P}^t$ be a Poincar\'e bundle on $\Jac^{d_t}\times X$. 
For large enough $n$ the sheaf
\[
 \mc{G}^0_{t}:= \mc{H}\opname{om}\left( \bigwedge^{r_t}Y^t\otimes \mc{O}_{\Jac^{d_t}},\pr_{\Jac^{d_t}*}(\mc{P}^t\otimes \pr_X^*\mc{O}_X(r_tn))\right)\,.
\]
is locally free and we set $\Gies^0_{t}:=\PP({\mc{G}^0_{t}}^\vee)$. Without loss of generality we may assume $\mc{O}_{\Gies^0_{t}}(1)$ to be very ample. We define $\Gies^0:=\prod_{t\in T}\Gies^0_{t}$.

On $J:=\Jac^{\tup{d}}\times \Jac^l$ with $\Jac^{\tup{d}}:=\prod_{t\in T}\Jac^{d_t}$ we consider the sheaf
 \[
  \mc{G}^1:=\mc{H}\opname{om}\left(Y^{\oplus\tup{\ka}}_{a_1,b_1}\otimes\mc{O}_{J},\pr_{J
*}\left(\pr_{\Jac^{\tup{d}}\times X}^* \mc{P}^{\otimes
\tup{\ka}c_1}\otimes\pr_{\Jac^l\times X}^*\mc{L}\otimes \pr_X^*\mc{O}_X(na_1) \right)\right)
 \]
 with
 \[
 \mc{P}^{\otimes \tup{\ka}}:=\bigotimes_{t\in T}(\pr_{\Jac^{d_t}}\times\id_X)^*{\mc{P}^{t}}^{\otimes \ka_t}\,.
\]
For $n$ sufficiently large this is a vector bundle and we set $\Gies^1:=\PP({\mc{G}^1}^\vee)$. Again, we may assume $\mc{O}_{\Gies^1}(1)$ to be very ample. Finally, we set $\Gies^2:=\PP\left(Y^{\oplus \tup{\ka}}_{a_2,b_2}\right)$ and define the \emph{Gieseker space} as
\[
 \Gies_n := \Gies^0\times_{\Jac^{\tup{d}}}\Gies^1 \times \Gies^2\,.
\]
Note that $\Gies_n$ is projective over $\Jac^{\tup{d}}\times \Jac^l$ and has a natural action of $\GL_T(Y)$. The universal family on $\QTmp_n$ determines an injective, $\GL_T(Y)$-invariant morphism
\[
 \gies_n:\QTmp_n\to \Gies_n
\]
over $\Jac^{\tup{d}}\times \Jac^l$.

\begin{remark}
 Let $L^{d_t}$ be a line bundle of degree $d_t$, $t\in T$, and $L$ a line bundle of degree $l$ on $X$. The fiber of $\Gies_n$ over the point in 
$\Jac^{\tup d}\times \Jac ^l$ determined by these line bundles is isomorphic to
 \[
 \prod_{t\in T} \Gies^0_{t}(L^{d_t})\times
\PP\left(\Hom\left(Y^{\oplus\tup{\ka}}_{a_1,b_1},H^0(L^{\tup{d}\otimes c_1}\otimes L(a_1n))
\right)^\vee\right)\times \PP(Y^{\oplus\tup{\ka}}_{a_2,b_2})\,.
 \]
 with $L^{\tup{d}}:=\bigotimes_{t\in T}L^{d_t}$ and
 \[
  \Gies^0_{t}(L^{d_t}):=\PP\left(\Hom\left(\bigwedge^{r_t}
Y^{t},H^0(L^{d_t}(n r_t))\right)^\vee\right)\,.
 \]
 Let $\la$ be a one-parameter subgroup of $\SL_T^{\tup{\ka}}(Y)$ with associated weighted flag $(Y_\bullet,\tup{\alpha})$. For a point $[M_t]\in \Gies^0_{t}(L^{d_t})$ one finds
\begin{equation}
 \begin{split} \label{eq:stab_in_Gies_0}
 \mu(\lambda,[M_t])=&-\sum_{k=1}^{\len(Y_\bullet)+1} \gamma_k\left(\rk(F^t_k)-\rk(F^t_{k-1}) \right)\\
  =&\sum_{j=1}^{\len(Y_\bullet)}\alpha_j \left(p(n)\rk(F^t_j)-\rk(E^t)\dim_{\tup{\ka}}(Y_j)
\right)\,,
 \end{split}
\end{equation}
where $p(n):=\dim_{\tup{\ka}}(Y)$ and $F^t_j\subset E^t$ is the subbundle generated by $Y^t_j$. For two points
\[
 [T_1]\in\PP\left(\Hom\left(Y^{\oplus\tup{\ka}}_{a_1,b_1},H^0(L^{\tup{d}\otimes
c_1}\otimes L(a_1n)) \right)^\vee\right)\,,\qquad [T_2]\in\PP(Y^{\oplus\tup{\ka}}_{a_2,b_2})\,,
\]
we get
\begin{align} \label{eq:stab_in_Gies_1}
 \mu(\lambda,[T_1])&=-\min\left\{\sum_{j=1}^{\len(Y_\bullet)}
\alpha_j(a_1\dim_{\tup{\ka}}(Y_j)-p(n)\nu_j(\tup{i}))\,\bigg|\,\tup{i}\in
I_1:T_{1|Y^{\otimes\tup{i}}}\neq 0 
\right\}\,, \\
\label{eq:stab_in_Gies_2}
 \mu(\lambda,[T_2])&=-\min\left\{\sum_{j=1}^{\len(Y_\bullet)}
\alpha_j(a_2\dim_{\tup{\ka}}(Y_j)-p(n)\nu_j(\tup{i}))\,\bigg|\,\tup{i}\in
I_2:T_{2|Y^{\otimes\tup{i}}}\neq 0 
\right\}
\end{align}
With $I_1$, $I_2$ and $Y^{\otimes \tup{i}}$ defined similarly as in Remark \ref{rem:tmp_stab}. In particular, we have the estimate
\begin{equation} \label{eq:GIT_stab_estimate}
  -a_2 \sum_{j=1}^{\len(Y_\bullet)}\alpha_j \dim_{\tup{\ka}}(Y_j)\le \mu(\lambda,[T_2])\le a_2\sum_{j=1}^{\len(Y_\bullet)}\alpha_j(p(n)-\dim_{\tup{\ka}}(Y_j))\,.
\end{equation}
\end{remark}

\subsection{Comparison of Stability Conditions}
We compare $(\delta_1,\delta_2,\tup{\ka},\tup{\chi})$-stability of a decorated tump with GIT-stability of a corresponding point in the parameter space. As an intermediate step, we introduce another notion of stability: For a vector bundle $E$ and a weighted flag $(E_\bullet,\tup{\alpha})$ we define
 \[
  M^{\textnormal{s}}_{\tup{\ka},\tup{\chi}}(E_\bullet,\tup{\alpha},n):= \sum_{j=1}^{\len(E_\bullet)} \alpha_j
\left(h^0_{\tup{\ka},\tup{\chi}}(E(n))\rk_{\tup{\ka}}(E_j)-h^0_{\tup{\ka},\tup{\chi}}(E_j(n))\rk_{
\tup{\ka}}(E) \right)\,.
 \]
\begin{definition}
 We call a decorated tump $(E,L,\phi,s)$ \emph{$(\delta_1,\delta_2,\tup{\ka},\tup{\chi},n)$-section-\sstable{}} if the condition
 \[
   M^{\textnormal{s}}_{\tup{\ka},\tup{\chi}}(E_\bullet,\tup{\alpha},n)
  +\delta_1\mu_1(E_\bullet,\tup{\alpha},\phi)+\delta_2\mu_2(E_\bullet,\tup{\alpha},s) (\ge) 0
 \]
holds for every weighted flag $(E_\bullet,\tup{\alpha})$ of $E$.
\end{definition}

\begin{proposition} \label{prop:section_stab<=>stab}
 There is an $n_1$ such that for all $n\ge n_1$ the following holds:
 \begin{enumerate}
  \item The notion of  $(\delta_1,\delta_2,\tup{\ka},\tup{\chi},n)$-section-\textup{(}semi\nobreakdash-\textup{)}stability and $(\delta_1,\delta_2,\tup{\ka},\tup{\chi})$-\textup{(}semi\nobreakdash-\textup{)}stability are equivalent.
  \item If $(E,L,\phi,s)$ is a $(\delta_1,\delta_2,\tup{\ka},\tup{\chi})$-semistable decorated tump, a weighted flag $(E_\bullet,\tup{\alpha})$ of $E$ satisfies
\begin{align} \label{eq:section-critical}
 & M^{\textnormal{s}}_{\tup{\ka},\tup{\chi}}(E_\bullet,\tup{\alpha},n)
  +\delta_1\mu_1(E_\bullet,\tup{\alpha},\phi)+\delta_2\mu_2(E_\bullet,\tup{\alpha},s) = 0
 \end{align}
 if and only if it is critical.
 \end{enumerate}
\end{proposition}
\begin{proof}
 (i) As in the proof of \cite{Beck2014DecSwamps}, Lemma 5.2, one can show that for $n$ large enough, any $(\delta_1,\delta_2,\tup{\ka},\tup{\chi},n)$-section-\sstable{} decorated tump $(E,L,\phi,s)$ satisfies $h^1(E^t(n))=0$ for all $t\in T$. Together with the Riemann--Roch theorem this implies $M^{\textnormal{s}}_{\tup{\ka},\tup{\chi}}(E_\bullet,\tup{\alpha},n) \le M_{\tup{\ka},\tup{\chi}}(E_\bullet,\tup{\alpha})$. Thus, $(E,L,\phi,s)$ is also $(\delta_1,\delta_2,\tup{\ka},\tup{\chi})$-(semi-)stable.
 
 We now show the converse: The proof of \cite{Beck2014DecSwamps}, Lemma 5.4, with the Le Potier--Simpson estimate replaced by Proposition \ref{prop:LS-estimate} shows that if $n$ is sufficiently large, for every $(\delta_1,\delta_2,\tup{\ka},\tup{\chi})$-\sstable{} decorated tump $(E,L,\phi,s)$ every $T$-split subbundle $F\subset E$ satisfies
\begin{equation} \label{eq:M^s>0}
  h^0_{\tup{\ka},\tup{\chi}}(E(n))\rk_{\tup{\ka}}(F)-
h^0_{\tup{\ka},\tup{\chi}}(F(n))\rk_{\tup{\ka}}(E)-\rk_{\tup{\ka}}(F)(a_1\delta_1+a_2\delta_2)> 0\,
\end{equation}
or $h^1(F^t)=0$ for all $t\in T$. Given a weighted flag $(E_\bullet,\tup{\alpha})$ we decompose the set of indices $I:=\{1,\ldots,\len(E_\bullet)\}$ into the two sets $I^B:=\{i\in I\,|\,\forall t\in T: h^1(E_i^t(n))=0\}$ and $I^A:=I\setminus I^B$. Denote the corresponding weighted flags by $(E_\bullet^{A/B},\tup{\alpha}^{A/B})$. If $I^A\neq \varnothing$, the estimate from Lemma 1.5.1.41, ii), in \cite{Schmitt08} gives
\begin{align*}
 &
M^{\textnormal{s}}_{\tup{\ka},\tup{\chi}}(E_\bullet,\tup{\alpha},n)+\delta_1\mu_1(E_\bullet,\tup{\alpha},
\phi)+\delta_2\mu_2(E_\bullet,\tup{\alpha},s)\\
 \ge\;&
M^{\textnormal{s}}_{\tup{\ka},\tup{\chi}}(E_\bullet^B,\tup{\alpha}^B,n)+\delta_1\mu_1(E_\bullet^B,\tup{
\alpha}^B,\phi)+\delta_2\mu_2(E_\bullet^B,\tup{\alpha}^B,s)\\
  & +
M^{\textnormal{s}}_{\tup{\ka},\tup{\chi}}(E_\bullet^A,\tup{\alpha}^A,n)- \sum_{1\le j\le \len(E_\bullet^A)} \alpha^A_j \rk_{\tup{\ka}}(E^A_j)(\delta_1a_1+\delta_2a_2)\,.
 \end{align*}
The last term is positive by the definition of $I^A$ and \eqref{eq:M^s>0}. Due to the definition of $I^B$ and the theorem of Riemann--Roch we have $M^{\textnormal{s}}_{\tup{\ka},\tup{\chi}}(E_\bullet^B,\tup{\alpha}^B,n)= M_{\tup{\ka},\tup{\chi}}(E_\bullet^B,\tup{\alpha}^B)$. Thus, $(E,L,\phi,s)$ is also $(\delta_1,\delta_2,\tup{\ka},\tup{\chi},n)$-section-\sstable{}.

(ii) Let $(E,L,\phi,s)$ be a $(\delta_1,\delta_2,\tup{\ka},\tup{\chi})$-\sstable{} decorated tump. The first paragraph of part (i) shows that a critical weighted flag $(E_\bullet,\tup{\alpha})$ of $E$ satisfies \eqref{eq:section-critical}. Let $(E_\bullet,\tup{\alpha})$ now be a weighted flag such that \eqref{eq:section-critical} holds. The arguments from the second paragraph of (ii) show $I^A=\varnothing$, i.e. $(E_\bullet,\tup{\alpha})=(E_\bullet^B,\tup{\alpha}^B)$. It follows that $(E_\bullet,\tup{\alpha})$ is critical.
\end{proof}

We linearize the action of $\SL^{\tup{\ka}}_T(Y)$ on $\Gies_n$ in the line bundle
\begin{equation*} 
 \mc{O}_{\Gies}(\eta_t,t\in T,\theta_1,\theta_2):= \left(\bigotimes_{t\in T}\pr_{\Gies^0_{t}}^*\mc{O}_{\Gies^0_{t}}(\eta_t)\right)\boxtimes \mc{O}_{\Gies^1}(\theta_1) \boxtimes  \mc{O}_{\Gies^2}(\theta_2)
\end{equation*}
with
\begin{equation} 
\begin{gathered} \label{eq:Linearization}
 \eta_t:=z\left(\ka_t(p(n)+(\tup{r},\tup{\chi})-a_1\delta_1-a_2\delta_2) -r\chi_t\right)  \,, \quad
 t\in T\,,\\
\theta_1:=zr\delta_1\,, \qquad \theta_2:=zr\delta_2\,.
\end{gathered}
\end{equation}
Here, $z$ is a natural number such that $\eta_t$, $t\in T$, $\theta_1$ and $\theta_2$ are positive integers.
\begin{remark}
  Note that our linearization is not modified by a character of $\SL^{\tup{\ka}}_T(Y)$ (compare \cite{Schmitt08}, \S 2.5.4).
\end{remark}

Let $p$ be a point in $\QTmp_n$ and denote by $(E,L,\phi,s)$ the corresponding decorated Tump.
\begin{proposition}
 Assume $a_2\delta_2<1$. Then for $n\ge n_1$ we have:
 \begin{enumerate}
  \item The point $\gies_n(p)$ is GIT \sstable{} with respect to the linearization given in \eqref{eq:Linearization} if and only if $(E,L,\phi,s)$ is $(\delta_1,\delta_2,\tup{\ka},\tup{\chi})$-\sstable{}.
  \item If $\gies_n(p)$ is GIT semistable, there is a bijection $\Gamma_p$ from the set of critical weighted flags $(E_\bullet,\tup{\alpha})$ to the set of critical one-parameter subgroups $\la$ of $\SL_T^{\tup{\ka}}(Y)$.
 \end{enumerate}
\end{proposition}

\begin{proof}
 (i) Given a weighted filtration $(E_\bullet,\tup{\alpha})$ of $E$, we construct a weighted flag of $Y$: Let $Y_\bullet$ be the flag induced by $U_j^t:=H^0(q_p^t(n))^{-1}H^0(E_j^t(n))$, $1\le j\le \len(E_\bullet)+1$, $t\in T$. For $1\le h\le \len(Y_\bullet)+1$ set $J(h):=\{j\,|\, U_j=Y_h \}$ and $\beta_h:=\sum_{j\in J(h)}\alpha_j$. We define $\Gamma_p(E_\bullet,\tup{\alpha}):=(Y_\bullet,\tup{\beta})$.
 
As in \cite{Beck2014DecSwamps}, Proposition 5.8, one shows that a one-parameter subgroup $\la$ of $\SL_T^{\tup{\ka}}(Y)$ with associated weighted flag $(Y_\bullet,\tup{\beta})=\Gamma_p(E_\bullet,\tup{\alpha})$ satisfies
\begin{equation} 
\label{eq:mu_le_M}
 \frac{\mu(\la,\gies_n(p))}{zp(n)}\le   M^{\textnormal{s}}_{\tup{\ka},\tup{\chi}}(E_\bullet,\tup{\alpha},n)+\delta_1\mu_1(E_\bullet,\tup{\alpha},
\phi)+\delta_2\mu_2(E_\bullet,\tup{\alpha},s) \,.
\end{equation}
Note that equality can only hold if $\len(E_\bullet)=\len(Y_\bullet)$ and $E_j(n)$ is generically generated by $Y_j$, $1\le j \le \len(E_\bullet)$. Details of this calculation can also be found in Proposition 6.41 in \cite{Beck2014}.

 Let now $(Y_\bullet,\tup{\beta})$ be a weighted flag of the $T$-split vector space $Y$. Let $F_h^t(n)\subset E^t(n)$ be the subsheaf generated by $Y^t_h$, $1\le h \le \len(Y_\bullet)$, $t\in T$, and let $E_\bullet$ be the flag induced by the $T$-split subbundles of $E$ generically generated by those $T$-split sheaves $F_h$ with $h^1(F^t_h(n))=0$ for all $t\in T$. Let $H(j):=\{h\,|\, F_h\textnormal{ generically generates }E_j \}$, $1\le j \le \len(E_\bullet)+1$, and $\alpha_j:=\sum_{h\in H(j)} \beta_h$, $1\le j \le \len(E_\bullet)$. We define $Q_p(Y_\bullet,\tup{\beta}):=(E_\bullet,\tup{\alpha})$.

Suppose $n\ge n_1$ and let $(E,L,\phi,s)$ be a $(\delta_1,\delta_2,\tup{\ka},\tup{\chi},n)$-section-\sstable{} decorated tump. Furthermore, let $\la$ be a one-parameter subgroup with associated weighted flag $(Y_\bullet,\tup{\beta})$. The proof of Proposition 5.9 in \cite{Beck2014DecSwamps} shows 
\begin{equation} 
\label{eq:mu_ge_M}
 \frac{\mu(\la,\gies_n(t))}{zp(n)} \ge 
 M^{\textnormal{s}}_{\tup{\ka},\tup{\chi}}(E_\bullet,\tup{\alpha},n)
 +\delta_1\mu_1(E_\bullet,\tup{\alpha},\phi)+\delta_2\mu_2(E_\bullet,\tup{\alpha},s)\,.
 \end{equation}
 for the weighted flag $(E_\bullet,\tup{\alpha}):=Q_p(Y_\bullet,\tup{\beta})$.
If equality holds, one has $\len(E_\bullet)=\len(Y_\bullet)$ and
\[
 Y^t_j=H^0(q_p(n))^{-1}H^0(E^t_j(n)) \,,  \text{ for all }t\in T\,, 1\le j \le \len(E_\bullet)\,.
\]
 In particular, one has $\Gamma_p(Q_p(Y_\bullet,\tup{\beta}))=(Y_\bullet,\tup{\beta})$ in this case. Again, details can be found in \cite{Beck2014}, Proposition 6.42.

Together with Proposition \ref{prop:section_stab<=>stab}, (i), the inequalities \eqref{eq:mu_le_M} and \eqref{eq:mu_ge_M} prove part (i). 

(ii) If $\la$ is a critical one-parameter subgroup for $p$, then by Proposition \ref{prop:section_stab<=>stab}, (ii), and \eqref{eq:mu_ge_M} the weighted flag $Q_p(Y_\bullet,\tup{\beta})$ is also critical. Thus, one finds $\Gamma_p(Q_p(Y_\bullet,\tup{\beta})=(Y_\bullet,\tup{\beta})$. Conversely, if $(E_\bullet,\tup{\alpha})$ is a critical flag, then Proposition \ref{prop:section_stab<=>stab}, (ii), and \eqref{eq:mu_le_M} show that any one-parameter subgroup $\la$ with associated weighted flag $(Y_\bullet,\tup{\beta}):=\Gamma_p(E_\bullet,\tup{\alpha})$ is critical. In particular, $E_j(n)$ is generated by $Y_j$. Finally, \eqref{eq:mu_ge_M} shows that $Q_p(\Gamma_p(E_\bullet,\tup{\alpha})$ is also critical and we have $Q_p(\Gamma_p(E_\bullet,\tup{\alpha})=(E_\bullet,\tup{\alpha})$.
\end{proof}

\subsection{Proof of Theorem \ref{thm:moduli_space_of_tumps}}
\label{subsec:proof_moduli_of_tumps}

The results of the previous sections show that the open subscheme
\[
 \QTmp_n^{\textnormal{(s)s}}:=\gies_n^{-1}(\Gies_n^{\textnormal{(s)s}})
\]
parametrizes $(\delta_1,\delta_2,\tup{\ka},\tup{\chi})$-\sstable{} decorated tumps. It is easy to see that the family $(\tilde{E},\tilde{\ka},\tilde{N}_1,\tilde{N}_2,\tilde{\phi},\tilde{s})$ on $\QTmp_n^{\textnormal{(s)s}}$ satisfies the local universal property and is compatible with the group action in the sense of \cite{Beck2014DecSwamps}, \S 2.2.
The following proposition shows that our notion of S-equivalence is correct.
\begin{proposition}
 Let $\la$ be a one-parameter subgroup of $\SL_T^{\tup{\ka}}(Y)$ with associated weighted flag $(Y_\bullet,\tup{\beta})$ and $p_\infty:=\lim_{t\to \infty}\la(t)\cdot p$ the limit point of $p$. Then, the admissible deformation of $(E,L,\phi,s)$ along $(E_\bullet,\tup{\alpha}):=Q_p(Y_\bullet,\tup{\beta})$ is isomorphic to the decorated tump defined by $p_\infty$.
\end{proposition}
\begin{proof}
We construct a certain family of decorated quotient tumps $(q_S,\ka_S,N_{1,S},N_{2,S},\phi_S,s_S)$ parametrized by $S:=\mathbb{A}^1$ as in \cite{Beck2014DecSwamps}, Proposition 6.4. The universal property of $\QTmp_n$ gives a morphism $f:\CC\to \QTmp_n$, which satisfies $\la(t)\cdot p=f(t^{-1})$ for $t\in \CC^*$. Finally, one checks that the decorated tump over $t=0$ is isomorphic to $\df_{(E_\bullet,\tup{\alpha})}(E,L,\phi,s)$. 
\end{proof}

It remains to show the existence of the good quotient.

\begin{proposition}
Suppose $a_2\delta_2<1$. Then, there is an $n_2$, such that for $n\ge n_2$ the restriction of the Gieseker morphism to the semistable locus
 \[
  \gies_{n|\QTmp_n^{\textnormal{ss}}}:\QTmp_n^{\textnormal{ss}}\to \Gies^{\textnormal{ss}}
 \]
is proper.
\end{proposition}
\begin{proof}
 We check the valuation criterion: Let $R$ be a discrete valuation ring and $K$ its quotient field. We consider a commutative diagram
 \[\xymatrix{
  \Spec(K) \ar[r]^f \ar[d]_p & \QTmp^{\textnormal{ss}}_n \ar[d]^{\gies_n^{\textnormal{ss}}}\\
  \Spec(R) \ar[r]^h &  \Gies_n^{\textnormal{ss}}\rlap{\,.}
  }
 \]
The morphism $f$ corresponds to a family of decorated quotients tumps on $\Spec(K)$. Since $\Quot_n$ is projective the family of $T$-split quotient sheaves extends to a quotient $q_R:Y\otimes\pr_X^*\mc{O}_X(-n)\to F_R$ on $\Spec(R)\times X$. Because $F_R$ may have torsion, we compose the morphism with $F\to E_R:=F_R^{\vee\vee}$. As a reflexive sheaf on the regular surface $\Spec(R)\times X$ the sheaf $E_R$ is locally free. Repeating the construction from Proposition \ref{prop:QTmp_exists} over $\Spec(R)$ instead of $\Quot_n$, we obtain a tuple
\[
 (q_R,\ka_R,N_{1,R},N_{2,R},\phi_R,s_R)
\]
over $\Spec(R)\times X$ inducing the morphism $h$. 

It remains to show that this is a family of decorated quotient tumps. Let $(q:Y\otimes \mc{O}_X(-n)\to E,L,\phi,s)$ be the restriction to the special fiber. Semistability in $\Gies_n$ implies $\dim_{\tup{\ka}}(\ker(H^0(q(n))))\le a_2\delta_2$. The assumption $a_2\delta_2<1$ thus ensures that $H^0(q(n))$ is injective. Let now $Q$ be a $T$-split quotient of $E$ of minimal slope and $U:=\ker(Y\to H^0(Q(n)))$. GIT-semistability with respect to the weighted flag $(U\subset Y,(1))$ and the Le Potier--Simpson estimate from Proposition \ref{prop:LS-estimate}, (i), yield a lower bound for $\mu_{\tup{\ka},\tup{\chi}}(Q)=\mu_{\min}(E)$. Hence, there is an $n_2$ such that for $n\ge n_2$ the bundle $E^t(n)$ is globally generated and $H^1(E^t(n))$ vanishes, $t\in T$. For dimensional reasons, $H^0(q(n))$ is an isomorphism and $q$ is surjective.
\end{proof}

\begin{proof}[Proof of Theorem \ref{thm:moduli_space_of_tumps}]
  The restriction of $\gies_n$ to the semistable locus is proper and injective, hence quasi-finite. According to \cite{EGAIV-3}, 8.11.1, it is also finite, in particular affine. Therefore, by \cite{Ramanathan1996b}, 5.1 Lemma, the projective good quotient $\QTmp_n^{\textnormal{ss}}/\!\!/\SL^{\tup{\ka}}_T(Y)$ and the geometric quotient $\QTmp_n^{\textnormal{s}}/\SL^{\tup{\ka}}_T(Y)$ exist. Now $\QTmp_n^{\textnormal{(s)s}}$ has the local universal property for \sstable{} decorated tumps, the action of $\GL_T(Y)$ is compatible and the notions of S-equivalence agree. The general theory of moduli spaces and GIT implies that these quotients are the desired moduli spaces.
\end{proof}

\subsection{Asymptotic Stability}
\label{subsec:asymptotic_stability}
In the following, we transfer the results on asymptotic stability of decorated swamps from \S4 in \cite{Beck2014AsymptoticStability} to the setting of decorated tumps.
\begin{definition}
 We call a decorated tump $(E,L,\phi,s)$ \emph{asymptotically $(\delta_2,\tup{\ka},\tup{\chi})$-\sstable{}} if for any weighted flag $(E_\bullet,\tup{\alpha})$ there exists a number $c_1\in \QQ_{>0}$ such that for all $\delta_1\ge c_1$ we have
 \[
  M_{\tup{\ka},\tup{\chi}}(E_\bullet,\tup{\alpha})+ \delta_1\mu_1(E_\bullet,\tup{\alpha},\phi) + \delta_2 \mu_2(E_\bullet,\tup{\alpha},s) (\ge) 0\,.
 \]
\end{definition}
\begin{remark} \label{rem:asympt_stability}
 A decorated tump $(E,L,\phi,s)$ is asymptotically $(\delta_2,\tup{\ka},\tup{\chi})$-\sstable{} if and only if  every weighted flag $(E_\bullet,\tup{\alpha})$ of $E$ satisfies
\begin{enumerate}
 \item \label{rem:asympt_stab_1}
    $\mu_1(E_\bullet,\tup{\alpha},\phi) \ge 0$ and 
 \item \label{rem:asympt_stab_2} 
    $\mu_1(E_\bullet,\tup{\alpha},\phi) =0\quad \Longrightarrow\quad  M_{\tup{\ka},\tup{\chi}}(E_\bullet,\tup{\alpha})+\delta_2\mu_2(E_\bullet,\tup{\alpha},s) (\ge) 0$.
\end{enumerate}
\end{remark}

\begin{proposition} \label{prop:stable=>asympt_stable}
 For fixed $\delta_2$, $\tup{\ka}$ and $\tup{\chi}$ there exists a constant $\Delta_1$, such that for all $\delta_1\ge \Delta_1$ a $(\delta_1,\delta_2,\tup{\ka},\tup{\chi})$-\sstable{} decorated tump is also asymptotically $(\delta_2,\tup{\ka},\tup{\chi})$-\sstable{}.
\end{proposition}
 \begin{proof}
 Let $(E,L,\phi,s)$ be a $(\delta_1,\delta_2,\tup{\ka},\tup{\chi})$-\sstable{} decorated tump. Condition \ref{rem:asympt_stab_2} of Remark \ref{rem:asympt_stability} obviously holds. Assume that condition \ref{rem:asympt_stab_1} is not satisfied. Let $\eta$ be the generic point of $X$, $K$ the function field of $X$ and set $\EE:=E_{|\eta}$ and $\EE_\rho:=E_{\rho|\eta}$. By assumption the point $[\phi_\eta]\in\PP(\EE_\rho)$ is unstable with respect to the action of
 $\SL_T^{\tup{\ka}}(\EE)$, and we let $\Lambda$ be an instability one-parameter subgroup (see \cite{RamRam84}, Def. 1.6). This one-parameter subgroup determines a weighted flag $(E_\bullet,\tup{\alpha})$ of $E$ with $\mu_0:=\mu_1(E_\bullet,\tup{\alpha},\phi)=\mu(\Lambda,[\phi_\eta])\le -1$. 

Let $W_\bullet$ be a $T$-split flag of $W=(\CC^{r_t},t\in T)$ of the same type as $E_\bullet$ and choose an open subset $U$ of $X$ and a trivialization $E_{|U}\cong W\times U$ with $E_{j|U}\cong W_j\times U$ for all $1\le j \le \len(E_\bullet)$. This induces an isomorphism $\SL_T^{\tup{\ka}}(\EE)\cong \SL_T^{\tup{\ka}}(W)\times \Spec(K)$ such that there exists a one-parameter subgroup $\la$ of $\SL_T^{\tup{\ka}}(W)$ with $\Lambda\cong \la\times \id_{\Spec(K)}$
 
Let $D$ be the torus of diagonal matrices in $\GL(W^{\oplus\tup{\ka}})$. The identification $X_*(D)\cong \bigoplus_{t\in T}\ZZ^{\ka_t r_t}$ and the standard pairing induce a pairing on $X_*(D)$. This defines a norm $\norm{\cdot}_{\tup{\ka}}$ on the set of one-parameter subgroups of $\SL_T^{\tup{\ka}}(W)$ (see \cite{Schmitt08}, \S 1.7.2). As in \cite{Beck2014AsymptoticStability}, Proposition 4.3, one finds 
 \[
  M_{\tup{\ka},\tup{0}}(E_\bullet,\tup{\alpha})\le -\frac{l\norm{\la}_{\tup{\ka}}^2}{\mu_0}\le l\norm{\la}_{\tup{\ka}}^2\,.
 \]
Since there are only finitely many conjugacy classes of possible instability one-parameter subgroups, we can find constants $C$, $C'$ and $C_{\tup{\chi}}$ such that
\begin{align*}
 C&\,\ge \norm{\la}_{\tup{\ka}}^2\,, \qquad C'\ge \sum_{j=1}^{\len(E_\bullet)} \alpha_j \left(\rk_{\tup{\ka}}(E)-\rk_{\tup{\ka}}(E_j)\right)\,,\\
 C_{\tup{\chi}}&\,\ge \sum_{j=1}^{\len(E_\bullet)} \alpha_j\left( \rk_{\tup{\chi}}(E)\rk_{\tup{\ka}}(E_j)-\rk_{\tup{\chi}}(E)\rk_{\tup{\ka}}(E_j)\right)\,.
\end{align*}
The $(\delta_1,\delta_2,\tup{\ka},\tup{\chi})$-semistability implies
\begin{align*}
    Cl+C_{\tup{\chi}} -\delta_1 + a_2\delta_2C' \ge 
   M_{\tup{\ka},\tup{\chi}}(E_\bullet,\tup{\alpha})+\delta_1\mu_1(E_\bullet,\tup{\alpha},\phi)+\delta_2\mu_2(E_\bullet,\tup{\alpha},s)\ge 0
\end{align*}
Thus, we get a contradiction for $\delta_1>\Delta_1:= Cl+C_{\tup{\chi}} + a_2\delta_2C' $.
 \end{proof}
 
\begin{proposition} \label{prop:asympt_bounded}
 There is a constant $C$ such that every asymptotically $(\delta_2,\tup{\ka},\tup{\chi})$-semistable decorated tump $(E,L,\phi,s)$ of type $(\tup{d},l)$ and every non-trivial $T$-split subbundle $F\subset E$ satisfies
 \[
  \mu_{\tup{\ka},\tup{\chi}}(F)\le \mu_{\tup{\ka},\tup{\chi}}(E)+C\,.
 \]
\end{proposition}
\begin{proof}
 Let $(E,L,\phi,s)$ be an asymptotically $(\delta_2,\tup{\ka},\tup{\chi})$-semistable decorated tump. Denote by $\eta$ the generic point if $X$ and let $\EE^t:=E^t_{|\eta}$, $t\in T$, and $\EE_\rho:=E_{\rho|\eta}$ be the restrictions. For $t\in T$ and $1\le r'<\rk(E^t)$ we consider the space
 \[
  P_{t,r'}:= \PP\left(\bigwedge^{\rk(E^t)-r'} \EE \right)\times_\eta \PP(\EE_{\rho})\,.
 \]
By Proposition 2.9 in \cite{Beck2014AsymptoticStability} there is an $n_0(t,r')$ such that for $n\ge n_0(t,r')$ a point $(x,y)\in P_{t,r'}$ which is unstable with respect to the linearization in $\mc{O}_{P_{t,r'}}(1,n)$, but where $y$ is semistable, every instability one-parameter subgroup $\la$ of $(x,y)$ satisfies $\mu(\la,y)=0$. We choose
\[
 n:=\max\{n_0(t,r')\,|\, t\in T, 1\le r'\le \rk(E^t)\}\,.
\]

Let now $t\in T$ be an index and $F$ a subbundle of $E^t$. Together with $\phi$ this defines a morphism
\[
 f: X \to \PP\left(\bigwedge^{\rk(E^t)-\rk(F)} E \right)\times_X \PP(E_{\rho})\,.
\]
If $f$ is generically semistable, the line bundle
\[
 f^* \mc{O}_{\PP\left(\bigwedge^{\rk(E^t)-\rk(F)} E \right)\times_X \PP(E_{\rho})}(1,n)=\det(E^t/F)\otimes L(-D)^{\otimes n}
\]
with some effective divisor $D$ on $X$ has non-negative degree, so that $\deg(F)\le \deg(E^t)+nl$. If $f$ is generically unstable, there is a one-parameter subgroup $\la$ of $\SL(\EE)$ inducing a weighted flag $(E_\bullet,\tup{\alpha})$ of $E$ with
\[
 \deg(E^t)-\deg(F)+nl+m_0M_{\tup{\ka},\tup{0}}(E_\bullet,\tup{\alpha})\ge 0\,.
\]
Due to our choice of $n$ we have $\mu_1(E_\bullet,\tup{\alpha},\phi)=0$, so asymptotic $(\delta_2,\tup{\ka},\tup{\chi})$-stability implies 
\[
 M_{\tup{\ka},\tup{\chi}}(E_\bullet,\tup{\alpha})+\delta_2\mu_2(E_\bullet,\tup{\alpha},s) \ge 0\,.
\]
Since there are only finitely many conjugacy classes of possible instability one-parameter subgroups, we can find constants $C_{\tup{\chi}}$ and $C'$ as in the proof of Proposition \ref{prop:stable=>asympt_stable}. Then, in any case we have
\[
 \deg(F)\le \deg(E^t)+nl -m_0( C+ a_2C'\delta_2)\,.
\]
The claim is an easy consequence of this.
\end{proof}

\begin{theorem} \label{thm:asymptotic_stability}
 There is constant $\Delta$ such that for all $\delta_1>\Delta$ a decorated tump of type $(\tup{d},l)$ is $(\delta_1,\delta_2,\tup{\ka},\tup{\chi})$-\sstable{} if and only if it is asymptotically $(\delta_2,\tup{\ka},\tup{\chi})$-\sstable.
\end{theorem}
\begin{proof}
 There is a finite set $U$ of tuples $(\tup{r}_t,t\in T,\tup{\alpha})$ with $l\in\NN$, $\tup{r}_t\in \NN^l$, $t\in T$, and $\tup{\alpha}\in \QQ_{>0}^l$,  such that $(\delta_1,\delta_2,\tup{\ka},\tup{\chi})$-(semi\nobreakdash-)stability of a decorated tump has to be checked for weighted flags $(E_\bullet,\tup{\alpha})$ with $(\tup{r}^t(E_\bullet),t\in T,\tup{\alpha})\in U$. We define
 \[
  C'':=\max\left\{\sum_{j=1}^{\len(\tup{\alpha})} \alpha_j \sum_{t\in T}\ka_t r_{t,j}\,\biggm|\,(\tup{r}_t,t\in T,\tup{\alpha})\in U \right\}\,.
 \]
 Let $C$ be the constant from Proposition \ref{prop:asympt_bounded}. We set
 \[
 M_0:= \max\left\{C\sum_{j=1}^{\len(\tup{\alpha})} \alpha_j r\sum_{t\in T}\ka_t r_{t,j}\,\biggm|\,(\tup{r}_t,t\in T,\tup{\alpha})\in U  \right\}\,.
 \]
 Finally let $m$ be an integer such that $mr\alpha$ is integral for all $(\tup{r}_t,t\in T,\tup{\alpha})\in U$.
 
Now suppose $\delta_1 >m(M_0+a_2\delta_2C'')$, let $(E,L,\phi,s)$ be an asymptotically $(\delta_2,\tup{\ka},\tup{\chi})$-\sstable{} decorated tump and $(E_\bullet,\tup{\alpha})$ a weighted flag whose type lies in $U$. If $\mu_1(E_\bullet,\tup{\alpha},\phi)=0$ holds, then $(\delta_1,\delta_2,\tup{\ka},\tup{\chi})$-(semi\nobreakdash-)stability with respect to $(E_\bullet,\tup{\alpha})$ follows from Remark \ref{rem:asympt_stability}, (ii). Otherwise, we have $\mu_1(E_\bullet,\tup{\alpha},\phi)\ge 1/m$ and hence
 \[
  M_{\tup{\ka},\tup{\chi}}(E_\bullet,\tup{\alpha})+\delta_1\mu_1(E_\bullet,\tup{\alpha},\phi)+\delta_2\mu_2(E_\bullet,\tup{\alpha},s)\ge -M_0+\frac{\delta_1}{m} -a_2\delta_2 C''> 0\,.
 \]
Together with Proposition \ref{prop:stable=>asympt_stable} this proves the claim for $\delta_1>\max\{\Delta_1,-m(M_0+a_2\delta_2C'')\}$.
\end{proof}

\section{Moduli of Decorated Principal Bundles}
\label{sec:Moduli_of_PB}
This section contains the construction of the moduli space of decorated principal bundles. The idea is to describe stable decorated principal bundles as asymptotically stable decorated tumps and then use the results from Section \ref{sec:Moduli_of_tumps}.

\subsection{Stable Decorated Principal Bundles}
Let $X$ be a smooth projective curve over the complex numbers, $x_0$ a point in $X$, $G$ an affine reductive group and $\sigma:G\to \GL(V)$ a representation.
\begin{definition}
 A \emph{$\sigma$-decorated principal $G$-bundle} is a pair $(\PB,s)$, where $\PB$ is a principal $G$-bundle on $X$ and $s$ is a point in the fiber over $x_0$ of the associated projective bundle $\PP(\PB_{\sigma})$.
\end{definition}
An isomorphism $f:(\PB,s)\to (\PB',s')$ of decorated principal bundles is an isomorphism $f:\PB\to \PB'$ with $f_\sigma(s')=s$, where $f_\sigma$ is the induced isomorphism $\PP(\PB'_\sigma)_{|x_0}\to \PP(\PB_\sigma)_{|x_0}$.

Since $G$ is affine, there is faithful representation $G\to \GL(W')$. We obtain an embedding $G\to \SL(W)$ with
\[
 W:= W'\oplus \left(\bigwedge^{\dim(W')}W'\right)^\vee\,.
\]
Since $G$ is also reductive, its radical $\rad(G)$ is a torus (see \cite[\S 11.21]{Bor91}). We thus find a decomposition $W=\bigoplus_{t\in T}W^t$ and characters $\chi_t$, $t\in T$, such that $g\cdot w=\chi_t (g)w$ for all $w\in W^t$ and $g\in \rad(G)$. Because the radical is contained in the center, the spaces $W^t$ are $G$-invariant. We get an embedding
\[
 \ka :G\to \bigoplus_{t\in T} \GL(W^t) \cap \SL(W)=:\SL_T(W)\,.
\]
Via $\ka$ the datum of a principal $G$-bundle $\PB$ is equivalent to the datum of a $T$-split vector bundle $E$ with $\rk(E^t)=\dim(W^t)$, $t\in T$ and $\bigotimes_{t\in T} \det(E^t)=\mc{O}_X$ together with a reduction of the structure group $\tau:X\to \Iso_T(W,E)/G$ with
\[
 \Iso_T(W,E):= \bigoplus_{t\in T}\Iso(W^t\otimes\mc{O}_x,E^t)\,.
\]
We call $\tup{d}=(\deg(E^t),t\in T)$ the \emph{type} of $\PB$.

Recall that the associated parabolic subgroup of a one-parameter subgroup $\la:\CC^*\to G$ is
\[
 Q_G(\la):=\{g\in G\,|\, \lim_{z\to \infty} \la(z)g\la^{-1}(z) \text{ exists in } G\}\,.
\]

\begin{definition}
 A \emph{weighted reduction} of a principal bundle $\PB$ is a pair $(\la,\beta)$, where $\la$ is a one-parameter subgroup of $G$ and $\beta:X\to \PB/Q_G(\la)$ is a reduction of the structure group.
\end{definition}

Let $(\PB,s)$ be a decorated principal bundle and $(\la,\beta)$ a weighted reduction. Fix a torus $T\subset G$ such that $\la\in X_*(T)$. The embedding $\ka$ induces a pairing on $X_*(T)$ and $X^*(T)$. There is a unique character $\chi_\la\in X^*(T)$ such that $\langle \chi_\la,\la'\rangle=(\la,\la')$ for all $\la'\in X_*(T)$. This character extends to a character on $Q_{\SL_T(\tup{r})}(\ka_*\la)$, which we also denote by $\chi_\la$. One can check, that $\chi_\la$ is independent of the choice of the torus $T$.

The section $\beta$ determines a trivialization $f:\PP(\PB_{\sigma})_{|x_0}\to \PP(V)$ up to an element $g\in Q_G(\la)$. We define
\[
 \mu(\la,\beta,s):=\mu_\sigma(\la,f(s))\,.
\]

\begin{definition}
 Let $\delta\in\QQ_{>0}$ and $\chi\in X^*(\SL_T(W))$. We call a $\sigma$-decorated principal bundle $(\PB,s)$ \emph{$(\delta,\chi)$-\sstable{}} if every weighted reduction $(\la,\beta)$ of $\PB$ satisfies
 \[
  \deg(\beta^*\PB\times^{\chi_\la} \CC) + \delta \mu(\la,\beta,s)+ \langle \chi,\ka_*\la\rangle (\ge) 0\,.
 \]
\end{definition}

\subsection{S-Equivalence of Decorated Principal Bundles}
Let $(\PB,s)$ be a $(\delta,\chi)$-semistable decorated principal bundle and $(\la,\beta)$  a weighted reduction. We denote by $\pi:Q_G(\la)\to L_G(\la)$ the projection to the Levi factor and by $i:L_G(\la)\to G$ the inclusion. Set $\PB':=i_*\pi_*\beta^*\PB$. The reduction $X\to P/Q_G(\la)\to \sigma_*P/Q_{\GL(V)}(\sigma_*\la)$ determines a flag $F_\bullet$ of $\PB_\sigma$, and we have
\[
 \PB'_\sigma =\bigoplus_{j=0}^{\len(F_\bullet)} F_{j+1}/F_{j}\,.
\]
Set $j_0:=\min\{1\le j \le \len(F_\bullet)\,|\, s_{|{F_j}_{|x_0}}\neq 0\}$, so that $s$ induces a point $s'\in \PP(F_{j_0}/F_{j_0-1})_{\{x_0\}}\subset \PP(\PB'_\sigma)_{\{x_0\}}$. We call
\[
 \df_{(\la,\beta)}(\PB,s):=(\PB',s')
\]
the \emph{admissible deformation} of $(\PB,s)$ along $(\la,\beta)$.

\begin{definition}
We call a weighted reduction $(\la,\beta)$ \emph{critical} if the condition
\[
  \deg(\beta^*\PB\times^\chi_\la \CC) + \delta \mu(\la,\beta,s)+ \langle \chi,\ka_*\la\rangle =0
\]
holds.
\end{definition}
\begin{definition}
 We define \emph{S-equivalence} as the equivalence relation generated by isomorphisms and
 \[
  (\PB,s)\sim_S \df_{(\la,\beta)}(\PB,s)
 \]
for all critical weighted reductions $(\la,\beta)$ of $\PB$.
\end{definition}
\subsection{Parametrized Families of Decorated Principal Bundles}
\begin{definition}
 A \emph{family of decorated principal bundles} of type $\tup{d}$ parametrized by a scheme $S$ is a triple $\mc{F}=(\PB_S,N_S,s_S)$, where $\PB_S$ is a principal bundle on $S\times X$, such that for every $s\in S$ the bundle $\PB_{S|\{s\}\times X}$ is of type $\tup{d}$, $N_S$ is a line bundle on $S$ and $s_S:\PB_{\sigma|S\times\{x_0\}}\to N_S$ is a surjective homomorphism.
 
 Two parametrized families $\mc{F}$ and $\mc{F}'$ are \emph{isomorphic} if there are isomorphisms $f:\PB_S\to\PB'_S$ and $h:N_S\to N'_S$ such that $s'_S\circ f_{\sigma|S\times\{x_0\}} = h\circ s_S$.
\end{definition}
If $\mc{F}$ is family of decorated principal bundles parametrized by $S$, then $\mc{F}$ defines an isomorphism class of decorated principal bundles $\mc{F}_s$ for every point $s\in S$. We call $\mc{F}$ \emph{$(\delta,\chi)$-\sstable{}} if $\mc{F}_s$ is so for all $s\in S$.

\begin{remark}
 By \cite{Schmitt08}, Theorem 1.1.6.1, there is a representation $\sigma':\GL_T(W)\to \GL(V')$, such that $\sigma$ is a direct summand of $\sigma'\circ \ka$. This induces an inclusion $\PP(\PB_\sigma)\subset \PP(\PB_{\sigma'\circ \ka})$. Without loss of generality we may thus assume that $\sigma$ is induced by a representation $\GL_T(W)\to \GL(V)$, which we also denote by $\sigma$. For technical reasons we have to assume that $\sigma$ is homogeneous. By Proposition \ref{prop:W_abc} there are numbers $a,b,c$ such that $V$ is a direct summand of $W_{a,b,c}$.
\end{remark}

\begin{theorem} \label{thm:moduli_space_of_dec_PB}
 Let $G\subset \GL_T(W)$ be an affine reductive group and $\sigma$ a direct summand of the natural representation $\GL_T(W)\to\GL(W_{a,b,c})$. For $a\delta<1$ the \textup{(}projective\textup{)} moduli space of $(\delta,\chi)$-\sstable{} $\sigma$-decorated principal bundles with structure group $G$ exists. 
\end{theorem}

The remainder of this section is devoted to the construction of this moduli space.
\subsection{Decorated Pseudo Principal Bundles}
Let $(\PB,s)$ be a decorated principal bundle. A weighted reduction $(\la,\beta)$ of $\PB$ determines a weighted flag $(E_\bullet,\tup{\alpha})$ of the associated $T$-split vector bundle $E$, such that
\[
 \deg(\beta^*\PB\times^{\chi_\la} \CC)=M_{\tup{\ka},\tup{0}}(E_\bullet,\tup{\alpha})
\]
with $\ka_t:=1$, $0_t:=0$ for $t\in T$.

A character $\chi$ of $\SL_T(W)$ is given by an element $\tup{\chi}\in \ZZ^T/\Delta$ where $\Delta$  is the diagonal. If $\la$ is a one-parameter subgroup of $\SL_T(W)$ with associated weighted flag $(W_\bullet,\tup{\alpha})$, one finds
\[
 \langle \chi,\la\rangle=\sum_{t\in T}\chi_t\sum_{j=1}^{\len(W_\bullet)} \alpha_j \left(\dim(W_j)
\dim(W^t)-\dim(W)\dim(W^t_j)\right)\,.
\]
A short calculation shows $\deg(\beta^*\PB\times^{\chi_\la} \CC)+\langle \chi,\la\rangle=M_{\tup{\ka},\tup{\chi}}(E_\bullet,\tup{\alpha})$. Finally we have $\mu(\la,\beta,s)=\mu_2(E_\bullet,\tup{\alpha},s)$. This implies the following lemma.
\begin{lemma} \label{lem:PB_stable<=>VB_stable}
A decorated principal bundle $(\PB,s)$ is $(\delta,\chi)$-\sstable{} if and only if the condition
 \[
 M_{\tup{\ka},\tup{\chi}}(E_\bullet,\tup{\alpha})+\delta\mu_2(E_\bullet,\tup{\alpha},s) (\ge) 0\,.
\]
holds for every weighted flag $(E_\bullet,\tup{\alpha})$ of the associated $T$-split vector bundle $E$ that is induced by a weighted reduction of $\PB$.
\end{lemma}

\begin{definition}
 A \emph{decorated pseudo principal bundle} is a triple $(E,\tau,s)$ consisting of a $T$-split vector bundle $E$ with $\tup{r}(E)=\tup{r}$ and trivial determinant, a non-trivial homomorphism
 \[
  \tau:\Sym\left(\,\bigoplus_{t\in T}{E^t}^\vee\otimes W^t \right)^G\to \mc{O}_X
 \]
and a point $s$ in $\PP(E_{\sigma|\{x_0\}})$.

An \emph{isomorphism} of decorated pseudo principal bundles is an isomorphism $f:E\to E'$ of $T$-split vector bundles such that $\tau\circ \tilde{f}=\tau'$ and $f_\sigma(s')=s$, where
\[
 \tilde{f}:=\Sym\left(\bigoplus_{t\in T}{f^t}^\vee\otimes \id_{W^t}\right)
\]
\end{definition}
\begin{definition}
A \emph{family of decorated pseudo principal bundles} parametrized by a scheme $S$ is tuple $\mc{F}=(E_S,N_S,\tau_S,s_S)$, where $E_S$ is a $T$-split vector bundle on $S\times X$ of rank vector $\tup{r}$, such that for each $s\in S$ the $T$-split bundle $E_{S|\{s\}\times X}$ has degree vector $\tup{d}$, $N_S$ is a line bundle on $S$,
\[
 \tau_S:\Sym\left(\bigoplus_{t\in T}{E^t_S}^\vee\otimes W^t \right)^G \to \pr_X^*\mc{O}_X
\]
is a homomorphism whose restriction to any point $s\in S$ is non-trivial, and $s_S:E_{S\sigma|S\times\{x_0\}}\to N_S$ is a surjective homomorphism.

Two families $\mc{F}$ and $\mc{F}'$ are \emph{isomorphic} if there are isomorphisms $f:E_S\to E'_S$ and $h:N_S\to N'_S$, such that $\tau'_S\circ \tilde{f}= \tau_S$ and $h\circ s_S=s'_S\circ f_{\sigma|S\times\{x_0\}}$. 
\end{definition}

A principal bundle $\PB_S$ on $S\times X$ corresponds to a $T$-split vector bundle $E_S$ and a reduction of the structure group $\beta_S:S\times X\to \Iso_T(W,E_S)/G$. We compose the morphism $\beta_S$ with the inclusion
\[
 \Iso_T(W,E_S)/G\to \Hom_T(W,E_S)/\!\!/G
\]
The resulting morphism is given by a homomorphism of $\mc{O}_X$-algebras
\[
 \tau_S(\beta_S):\Sym\left(\,\bigoplus_{t\in T}{E^t}^\vee\otimes W^t\right)^G \to \mc{O}_X\,.
\]
\begin{definition} 
 Let $(\PB_S,N_S,\phi_S,s_S)$ be a family of decorate principal bundles and $E$ the associated $T$-split vector bundle of $\PB$. Then we call $(E_S,N_S,\tau_S(\beta_S),s_S)$ the \emph{associated family of decorated pseudo principal bundles}.
\end{definition}
\begin{lemma} \label{lem:pseudo_PB_is_PB}
 A decorated pseudo principal bundle $(E,\tau,s)$ definig the section $\sigma:X\to \Hom_T(W,E)/\!\!/G$ is the image of a decorated principal bundle if and only if there is a point $x\in X$ such that $\sigma(x)$ lies in $\Iso_T(W,E)/G$.
\end{lemma}
\begin{proof}
 See \cite{Schmitt08}, Lemma 2.6.3.1.
\end{proof}

\subsection{Decorated Pseudo Principal Bundles as Decorated Tumps}
Consider the space 
\[
 H:=\bigoplus_{t\in T}\End(W^t)
\]
 with the natural left action of $\GL_{T}(W)$ and the natural right action of $G$. Since these actions commute, there is an induced action of $\GL_T(W)$ on $\PP(H^\vee)/\!\!/G=\Proj(\Sym(H^\vee)^G)$. Lemma 2.6.2.1 in \cite{Schmitt08} tells us:
\begin{lemma} \label{lem:iso_<=>_stable}
 Let $h\in H$ be a point and $[h]\in \PP(H^\vee)/\!\!/ G$ the induced point.
 \begin{enumerate}
  \item If $h$ is an isomorphism, then $[h]$ is $\SL_T(W)$-polystable.
  \item If $[h]$ is $\SL_T(W)$-semistable, then $h$ is an isomorphism.
 \end{enumerate}
\end{lemma}

Let $d\in \NN$ be a number such that $\Sym(H^\vee)^G$ is generated in degree less or equal to $d$ and the graded ring 
\[
 \Sym^{(d!)}(H^\vee)^G:= \bigoplus_i \Sym^{d!i}(H^\vee)^G
\]
is generated in degree one. Then the vector space homomorphism
\[
 V_1:=\bigoplus_{\substack{\tup{s}\in\NN^d\\ \sum i s_i=d!}} \bigotimes_{i=1}^d
\Sym^{s_i}(\Sym^{i}(H^\vee)^G)\to \Sym^{d!}(H^\vee)^G
\]
is surjective. Note that the representation $\rho:\GL_T(W)\to \GL(V_1)$ is homogeneous of degree $d!$.

By choice of $d$ the vertical homomorphisms in the commutative diagram
 \[
  \xymatrix{
  \Sym(V_1) \ar[r] \ar@{->>}[d] & \Sym\left(\bigoplus_{i=1}^d \Sym^{i}(H^\vee)^G\right) \ar@{->>}[d] \\
  \Sym^{(d!)}(H^{\vee})^G \ar[r] &  \Sym(H^\vee)^G  
  }
 \]
are surjective. The horizontal induce isomorphisms of the corresponding (weighted) projective spaces.
\begin{lemma} \label{lem:mu>0<=>mu>0}
 Let $[h]$ be a point in $\PP(H^\vee)/\!\!/G$ and $v[h]$ its image in $\PP(V_1)$. For $\la\in X_*(\SL_T(W))$ we find
\begin{align*}
  \mu(\la,[h])(\ge) 0  && \Longleftrightarrow && \mu(\la,v[h]) (\ge) 0\,.
\end{align*}
\end{lemma}
\begin{proof}
 Let $[h]$ be represented by $h_i:\Sym^i(H^\vee)^G\to \CC$, $i=1,\ldots,d$. Then $v[h]$ is determined by the linear forms
\begin{align*}
  h_{\tup{s}}\colon \bigotimes_{i=1}^d \Sym^{s_i}(\Sym^{i}(H^\vee)^G) &\to \CC \\
   (u_1^{(1)}\cdot \cdots \cdot u_{s_1}^{(1)})\otimes \cdots \otimes(u_1^{(d)}\cdot \cdots \cdot u_{s_d}^{(d)}) &\mapsto \prod_{i=1}^d\prod_{j=1}^{s_i} h_i(u^{(i)}_j)\,.
 \end{align*}
This implies the claim.
\end{proof}

Let $(E_S,N_S,\tau_S,s_S)$ be a parametrized family of decorated pseudo principal bundles. We consider the bundle
\[
 E_{S,\rho}=\bigoplus_{\substack{\tup{s}\in\NN^d\\ \sum i s_i=d!}} \bigotimes_{i=1}^d
\Sym^{s_i}\left(\Sym^{i}\left(\bigoplus_{t\in T} E_S^{t\vee}\otimes
W^t\right)^G\right)\,.
\]
The composition of $\tau_S^{d!}$ with the surjection
 \[
E_{S,\rho}\to \Sym^{d!}\left(\bigoplus_{t\in T} E_S^{t\vee}\otimes
W^t\right)^G
\]
defines a non-trivial homomorphism $\phi_S:E_{S,\rho}\to \pr_X^*\mc{O}_X$. 

\begin{definition}
 We call $(E_S,\mc{O}_X,N_S,\phi_S,s_S)$ the \emph{family of decorated tumps associated with the family $(E_S,N_S,\tau_S,s_S)$} of decorated pseudo principal bundles.
\end{definition}

\begin{lemma} \label{lem:PPB_to_tump_injective}
 The map taking the isomorphism class of a decorated pseudo principal bundle $[(E,\tau,s)]$ to the isomorphism class of its associated decorated tump $[(E,\mc{O}_X,\phi,s)]$ is injective.
\end{lemma}
\begin{proof}
 The proof is the same as that of Proposition 2.6.3.2 in \cite{Schmitt08}.
 Let $(E,\tau,s)$ and $(E',\tau',s')$ be two decorated pseudo principal bundles with isomorphic associated decorated tumps. It suffices to consider the case $E=E'$, $s=s'$ and $\phi(\tau)=\phi(\tau')$. Note that $\tau$ and $\tau'$ are determined by the induced homomorphisms 
 \[
  \tau_j,\tau'_j:\Sym^j\left(\bigoplus_{t\in T}{E^t}^\vee\otimes W^t\right)^G\to \mc{O}_X\,, \quad j=1,\ldots,d\,.
 \]
For each $j$ there exists $l_j\in \NN$ with $l_j j=d!$. By assumption we have $\Sym^{l_j}(\tau_i)=\Sym^{l_j}(\tau'_j)$. Hence, there is a $l_j$-th root of unity $\zeta_j\in \CC^*$ such that $\tau_j=\zeta_j\tau'_j$. It remains to show that there is a $d!$-th root of unity $\zeta$ with $\zeta_j=\zeta^j$.

Let $\EE=E_\eta$ be the restriction to the generic point. By assumption the points $[\tau_{1|\eta},\ldots,\tau_{d,|\eta}]$ and $[\tau'_{1|\eta},\ldots,\tau'_{d|\eta}]$ have the same image under the Veronese map
\[
 \PP\left(\bigoplus_{j=1}^d \Sym^j\left({E^t}^\vee\otimes W^t\right)^G \right)\to \PP(\EE_\rho)\,.
\]
Since this map is well known to be injective, the claim follows.
\end{proof}

\begin{lemma}
 A decorated tump $(E,\mc{O}_X,\phi,s)$ is the image of a decorated principal bundle if and only if the condition $\mu_1(E_\bullet,\tup{\alpha},\phi)\ge 0$ holds for all weighted flags $(E_\bullet,\tup{\alpha})$ of $E$. 
\end{lemma}
\begin{proof}
 This follows from Lemmata \ref{lem:iso_<=>_stable} and \ref{lem:mu>0<=>mu>0} together with Lemma \ref{lem:pseudo_PB_is_PB}.
\end{proof}

\begin{definition}
 We say that a family of decorated pseudo principal bundles is \emph{$(\delta,\tup{\chi})$-\sstable{}} if its associated family of decorated tumps is asymptotically $(\delta,\tup{\ka},\tup{\chi})$-\sstable{} with $\ka_t=1$, $t\in T$.
 
 Two decorated pseudo principal bundles are \emph{S-equivalent} if their associated isomorphism classes of decorated tumps are S-equivalent.
\end{definition}

\subsection{Comparison of Stability Notions}
Finally, we need to compare the notion of stability of decorated principal bundles with that of decorated pseudo principal bundles.

\begin{proposition} \label{prop:BMRT}
 Let $K$ be a perfect field, $G$ an affine reductive group and $H\subset G$ a closed reductive subgroup over $K$. For every point $[g]\in G/H$ and every one-parameter subgroup $\la:K^*\to G$, such that the limit $[g]_\infty:=\lim_{c\to\infty} \la(c)\cdot[g]$ exists in $G/H$, there is an element $g$ in the radical of $Q_G(\la)$ with $[g]_\infty=g'\cdot [g]$.
\end{proposition}
\begin{proof}
 Kraft and Kuttler proved this in the case that $K$ is algebraically closed. The statement for algebraically closed $K$ is Theorem 3.3 in \cite{BMRT13}.
\end{proof}

\begin{lemma} \label{lem:mu(la,x)=0<=>la_in_G}
 Let $K$ be a perfect field over $\CC$ and $\la:K^*\to \SL_T(W\otimes K)$ a one-parameter subgroup. For $x:=[(\id_{W^t},t\in T)]\in \PP(H^\vee\otimes K)/\!\!/G(K)$ the following are equivalent:
 \begin{enumerate}
  \item There is a one-parameter subgroup $\la'$ of $G(K)$ and an element $g\in Q_{\SL_T(W\otimes K)}(\la)$  with $\ka_*\la'=g^{-1}\la g$.
  \item $\mu(\la,x)=0$.
 \end{enumerate}
\end{lemma}
\begin{proof}
 The implication (i)$\Rightarrow$(ii) is obvious because $G$ stabilizes $[\id_W]\in H/\!\!/G $. Now suppose $\mu(\la,x)=0$ and let $x_\infty:=\lim_{c\to \infty}\la(c)\cdot x$. Then by Lemma \ref{lem:iso_<=>_stable} there exists $g\in \GL_T(W\otimes K)$ such that $[g]=x_\infty$. By Proposition \ref{prop:BMRT} we can choose $g$ in the radical of $Q_{\SL_T(W\otimes K)}(\la)$. Since $\la$ leaves $x_\infty$ invariant, $\la$ factors through the stabilizer $g\ka(G(K))g^{-1}$
\end{proof}

\begin{proposition} \label{prop:la_in_G<=>mu_1=0}
 Let $(\PB,s)$ be a decorated principal bundle, $(E,\mc{O}_X,\phi,s)$ the associated decorated tump and $(E_\bullet,\tup{\alpha})$ a weighted flag. Then the following are equivalent:
 \begin{enumerate}
  \item There is a weighted reduction $(\la,\beta)$ of $\PB$ inducing $(E_\bullet,\tup{\alpha})$.
  \item $\mu_1(E_\bullet,\tup{\alpha},\phi)=0$.
 \end{enumerate}
\end{proposition}
 This is Proposition 2.6.3.4 in \cite{Schmitt08}. Using Proposition \ref{prop:BMRT} we can slightly simplify  the proof.
\begin{proof}
 The implication (i)$\Rightarrow$(ii) follows from the corresponding statement in Lemma \ref{lem:mu(la,x)=0<=>la_in_G}. Now suppose that $\mu_1(E_\bullet,\tup{\alpha},\phi)=0$. Let $f:U\to X$ be an fppf-morphism and $h:f^*\PB\to U\times G$ a trivialization. Let $\eta$ be the generic point of $U$ and $K$ its function field. Then, $h$ also induces trivializations $f^*E\cong W\otimes \mc{O}_U$ of $T$-split vector bundles and $f^*\Aut(\PB)\cong G(K)$. We consider the $T$-split $K$-vector space $\EE:=f^*E_{|\eta}$. 
 Choose a one-parameter subgroup $\Lambda:K^*\to \SL_T(\EE)$ with associated weighted flag $(E_{\bullet|\eta},\tup{\alpha})$. By assumption we have $\mu(\Lambda,f^*\phi(\tau))=0$, so that by Lemma \ref{lem:mu(la,x)=0<=>la_in_G} there is a one-parameter subgroup $\Lambda'$ of $G(K)$ inducing the same weighted flag. Finally, there are a one-parameter subgroup $\la:\CC^*\to G$ and an element $g\in G(K)$ such that $\la\times \id_{\Spec(K)}=g\cdot \Lambda'\cdot g^{-1}$.
 
 The flag $E_\bullet$ determines a reduction $\beta':X\to \Iso_T(W,E)/Q_{\GL_T(W)}(\la)$. Let us consider the trivialization $h':=g\cdot h$. Then $h'\circ\beta'_{|\eta}=[\id_W\otimes \id_K]$ lies in $G(K)/Q_{G(K)}(\la)$ and thus $\beta$ extends to a reduction $\beta:X\to \PB/Q_{G}(\la)$.
\end{proof}

\begin{proposition}[Semistable Reduction]
A decorated pseudo principal bundle is $(\delta,\tup{\chi})$-\sstable{} if and only if it comes from a $(\delta,\chi)$-\sstable{} decorated principal bundle.
\end{proposition}
\begin{proof}
 Suppose $(E,\tau,s)$ is a $(\delta,\tup{\chi})$-\sstable{} decorated pseudo principal bundle. By part (i) of Remark \ref{rem:asympt_stability} and Lemma \ref{lem:pseudo_PB_is_PB}, $(E,\tau,s)$ is the image of a decorated principal bundle $(\PB,s)$. The implication (i)$\Rightarrow$(ii) in Proposition \ref{prop:la_in_G<=>mu_1=0}, part (ii) of Remark \ref{rem:asympt_stability} and Lemma \ref{lem:PB_stable<=>VB_stable} show that $(\PB,s)$ is $(\delta,\chi)$-\sstable{}.

 Now suppose $(\PB,s)$ is a $(\delta,\chi)$-\sstable{} decorated principal bundle, $(E,\tau,s)$ the associated decorated pseudo principal bundle and $(E_\bullet,\tup{\alpha})$ a weighted flag of $E$. Because of Lemma \ref{lem:pseudo_PB_is_PB} we have $\mu_1(E_\bullet,\tup{\alpha},\phi(\tau))\ge 0$. If $\mu_1(E_\bullet,\tup{\alpha},\phi(\tau))=0$ holds, then by Proposition \ref{prop:la_in_G<=>mu_1=0} there exists a weighted reduction $(\la,\beta)$ of $\PB$ which induces $(E_\bullet,\tup{\alpha})$. By Lemma \ref{lem:PB_stable<=>VB_stable} we have
 \[
 M_{\tup{\ka},\tup{\chi}}(E_\bullet,\tup{\alpha})+ \delta \mu_2(E_\bullet,\tup{\alpha},s) (\ge) 0\,.
 \]
Thus, $(E,\tau,s)$ is $(\delta,\tup{\chi})$-\sstable{}.
 
\end{proof}

\begin{corollary}
The category of $(\delta,\tup{\chi})$-\sstable{} decorated pseudo principal bundles is equivalent to the category of $(\delta,\chi)$-\sstable{} decorated principal bundles.
\end{corollary}

\subsection{Construction of the Parameter Space}
 As in the case of tumps, the parameter space is constructed as the fine moduli space of another moduli functor.  Let $Y^t$ be a complex vector space of dimension $p_t(n):=d_t+r_t(n+1-g)$, $t\in T$, and set $Y:=(Y^t,t\in T)$.

\begin{definition}
 A \emph{family of decorated quotient pseudo principal bundles} parametrized by a scheme $S$ is tuple $(q_S,N_S,\tau_S,s_S)$, where $q_S:Y\otimes\pr_X^*\mc{O}_X(-n)\to E_S$ is a quotient of $T$-split vector bundles with $\tup{r}(E)=\tup{r}$, such that
 \[
  \pr_{S*}(q_S\otimes\id_{\pr_X^*\mc{O}_X(n)}):Y\otimes \mc{O}_S\to \pr_{S*}(E_S\otimes \pr_X^*\mc{O}_X(n))
 \]
is an isomorphism, and $(E_S,N_S,\tau_S,s_S)$ is a family of decorated pseudo principal bundles.
\end{definition}

\begin{proposition} \label{prop:QPPB_exists}
 The fine moduli space $\QPPB_n$ of decorated quotient pseudo principal bundles exists.
\end{proposition}
\begin{proof}
 Let $\Quot_n$ be the scheme from Proposition \ref{prop:QTmp_exists} with the universal quotient of $T$-split vector bundles $q:Y\otimes\pr_X^*\mc{O}_X(-n)\to Q$. The determinant of the total sheaf $Q_{\textnormal{tot}}:=\bigoplus_{t\in T} Q^t$ determines a morphism $\Quot_n\to \Jac^0$. Let $P_1$ be the fiber over the point corresponding to the trivial bundle $\mc{O}_X$. Then there is a line bundle $\mc{A}$ on $P_1$, such that $\det(Q_{\textnormal{tot}})\cong \pr_{P_1}^*\mc{A}$. Using the canonical isomorphism
 \[
  Q_{\textnormal{tot}}^\vee = \det(Q_{\textnormal{tot}})^\vee\otimes \bigwedge^{r-1}Q_{\textnormal{tot}}\,.
 \]
the surjection $Y\otimes\pr_X^*\mc{O}_X(-n)\to Q$ on $P_1\times X$ induces a surjection
\[
 \Sym\left(W\otimes
\pr_{P_1}^*\mc{A}^\vee \otimes\bigwedge^{r-1}\bigl(Y_{\textnormal{tot}}\otimes\pr_X^*(\mc{O}_X(-n))\bigr)\right)^G \to
 \Sym\left(\bigoplus_{t\in T} W^t\otimes {Q^{t}}^\vee\right)^G
\]
with $Y_{\textnormal{tot}}:=\bigoplus_{t\in T} Y^t$. Over $P_1$ we consider the sheaf
\[
 P_2:=\bigoplus_{i=1}^d \Hom\left(\Sym^{i}\left(W\otimes
\pr_{P_1}^*\mc{A}^\vee\otimes \bigwedge^{r-1}Y_{\textnormal{tot}}\right)^{\!\!G}\!\!,\,\pr_{P_1*}\pr^*_X\mc{O}_X(i n(r-1)) \right)\,.
\]
For large enough $n$ this sheaf is locally free and the evaluation maps
\[
 \ev_i:H^0(X,\mc{O}_X(i n(r-1)))\otimes\mc{O}_{P_2\times X}\to \pr_X^*\mc{O}_X(i n(r-1))\,,\qquad i=1,\ldots,d
\]
 are surjective. We compose these with the tautological homomorphisms to construct the homomorphisms
\[
 \phi_i\colon \Sym^{i}\left(W\otimes
\pr_{P_1}^*\mc{A}^\vee \otimes \bigwedge^{r-1}Y_{\textnormal{tot}} \right)^G\!\!\!\! \to \,\pr_X^*\mc{O}_X(in(r-1))\,,\quad i=1,\ldots,d 
\]
on $P_2\times X$. We tensor each $\phi_i$ with $\id_{\pr_X^*\mc{O}_X(-in(r-1))}$ and consider the homomorphism
\[
 \phi\colon \mc{W}:=\bigoplus_{i=1}^d\Sym^{i}\left(W\otimes\pr_{P_1}^*\mc{A}^\vee \otimes\bigwedge^{r-1}\bigl(Y_{\textnormal{tot}}\otimes
\pr^*_X\mc{O}_X(-n)\bigr) \right)^G \to \mc{O}_{P_2\times X}\,.
\]
This determines a homomorphism of $\mc{O}_{P_2\times X}$-algebras
\[
 \tau_2\colon \Sym(\mc{W})\to \mc{O}_{P_2\times X}\,,
\]
By our choice of $d$ we also have a surjection of graded $\mc{O}_{P_2\times X}$-algebras
\[
 \Sym(\mc{W})\to \Sym\left(\bigoplus_{t\in T}W^t \otimes
{\pr_{\Quot_n}^*Q^{t}}^\vee \right)^G\,.
\]
Let $P_3\subset P_2$ be the closed subscheme, such that $\tau_2$ factors through this surjection over $P_3\times X$. We obtain the homomorphism
\[
 \tau_3\colon \Sym\left(\bigoplus_{t\in T}W^t \otimes \pr_{\Quot}^*{Q^t}^\vee\right)^G\to
\mc{O}_{P_3\times X}\,.
\]
Let $\QPPB_n$ be the associated projective bundle $\PP(\pr_{\Quot_n}^*Q_{\sigma})_{|P_3\times\{x_0\}}$ over $P_3$. On $\QPPB_n \times X$ we have the quotient $\tilde{q}:=\pr_{\Quot_n}^*q$ of $T$-split vector bundles and the homomorphism $\tilde{\tau}:=\pr_{P_3\times X}^*\tau_3$. Together with the tautological homomorphism $\tilde{s}:\tilde{E}_\sigma\to \tilde{N}:=\mc{O}_{\PP(\QPPB_n)}(1)$ we have a universal family $(\tilde{q},\tilde{N},\tilde{\tau},\tilde{s})$ of decorated quotient pseudo principal bundles. 
\end{proof}

\subsection{Proof of Theorem \ref{thm:moduli_space_of_dec_PB}}

Let $\QTmp_n$ be the moduli space of decorated quotient tumps from Proposition \ref{prop:QTmp_exists} and $\QPPB_n$ the moduli space of decorated quotient pseudo principal bundles from Proposition \ref{prop:QPPB_exists}. The associated family of decorated quotient tumps determines a $\GL_T(Y)$-equivariant morphism $f:\QPPB_n\to \QTmp_n$. We let 
\[
\QPPB_n^{(\delta,\tup{\chi})\textnormal{-(s)s}}:=f^{-1}\left(\QTmp_n^{(\Delta,\delta,\tup{\ka},\tup{\chi})\textnormal{-(s)s}}\right)
\]
denote the open subscheme of \sstable{} objects.

By Proposition \ref{prop:asympt_bounded} the class of asymptotically $(\delta,\tup{\ka},\tup{\chi})$-semistable decorated tumps is bounded. Hence, there exists an $n_0$ such that for every $n\ge n_0$ and every decorated pseudo principal bundle $(E,\tau,s)$ the $T$-split bundle $E(n)$ is globally generated and $H^1(E^t(n))$ vanishes for all $t\in T$. An easy consequence is the following lemma.
\begin{lemma}
 The family $(\tilde{E},\tilde{N},\tilde{\tau},\tilde{s})$ on $\QPPB_n^{(\delta,\tup{\chi})\textnormal{-(s)s}}$ satisfies the local universal property for families of $(\delta,\tup{\chi})$-\sstable{} decorated pseudo principal bundles.
\end{lemma}

There is a natural action of $\GL_T(Y)$ on $\QPPB_n$. As in the case of tumps one checks the following:
\begin{lemma}
 Let $f_1,f_2:S\to \QPPB_n^{(\delta,\tup{\chi})\textnormal{-(s)s}}$ be two morphisms. Then the pullbacks of the locally universal family are isomorphic if and only if there exists a morphism $g:S\to \GL_T(Y)$ such that $g\cdot f_1=f_2$.
\end{lemma}

 To finish the proof it suffices to construct the $\GL_T(Y)$-quotient of $\QPPB_n^{(\delta,\tup{\chi})\textnormal{-(s)s}}$. 
\begin{proof}[Proof of Theorem \ref{thm:moduli_space_of_dec_PB}]
 Since $\CC^*$ acts trivially on $\Quot_n$, we can construct the weighted projective space $Z:=\QPPB_n/\CC^*$ over $\Quot_n$. Because $\CC^*$ acts trivially on $\QTmp_n$, the morphism $f$ factors through a morphism $\bar{f}:Z \to \QTmp_n$ over $\Quot_n$. This morphism is proper and by Lemma \ref{lem:PPB_to_tump_injective} also injective. Thus, by \cite{EGAIV-3}, 8.11.1, it is finite and in particular affine. The same is true for the restriction
 \[
  \bar{f}:\QPPB_n^{(\delta,\tup{\chi})\textnormal{-(s)s}}/\CC^*=\bar{f}^{-1}\left(\QTmp_n^{(\Delta,\delta,\tup{\ka},\tup{\chi})\textnormal{-(s)s}}\right)\to
  \QTmp_n^{(\Delta,\delta,\tup{\ka},\tup{\chi})\textnormal{-(s)s}}\,.
 \]
If we suppose $a\delta<1$, then the proof of Theorem \ref{thm:moduli_space_of_tumps} shows that the (projective) $\SL_T(Y)$-quotient of the right-hand side exists. By \cite{Ramanathan1996b}, 5.1 Lemma, the (projective) GIT-quotient 
\[
 \QPPB_n^{(\delta,\tup{\chi})\textnormal{-(s)s}}/\!\!/\GL_T(Y)=(\QPPB_n^{(\delta,\tup{\chi})\textnormal{-(s)s}}/\CC^*)/\!\!/\SL_T(Y)
\]
also exists.
\end{proof}

\section{Examples}
\label{sec:examples}
We show that the construction from Section \ref{sec:Moduli_of_PB} specializes to parabolic bundles or principal bundles with a level structure for certain choices of $\sigma$.
\subsection{Parabolic Principal Bundles}
Let $Q\subset P$ be a parabolic subgroup and $w$ a rational one-parameter subgroup of $G$ with $Q_G(w)=Q$.
\begin{definition}
 A \emph{parabolic principal bundle} is a principal bundle $\PB$ on $X$ together with a point $s\in \PB/Q_{|x_0}$.
\end{definition}

We consider the embedding $\iota:=\iota_{\tup{\ka}}\circ \ka:G\to \GL(W_{\textnormal{tot}})$. The one-parameter subgroup $\iota_*w$ defines a weighted flag $(W_\bullet,\tup{\beta})$ of $W_{\textnormal{tot}}$. Let $l$ be the length and $\tup{r}$ the type of this flag. Then the embedding $\iota$ also induces a closed embedding of $G/Q$ into the flag variety $ \Fl(W_{\textnormal{tot}},\tup{r})\cong \GL(W_{\textnormal{tot}})/Q_{\GL(W_{\textnormal{tot}})}(i_*w)$. Using the Pl\"ucker and the Segre embedding we obtain an embedding $i:G/Q\to \PP(V)$ with
\[
 V:=\left(\bigotimes_{i=1}^{l} \bigwedge^{r-r_i} W_{\textnormal{tot}}\right)^{\otimes z \beta_i}\,.
\]
Here, $z$ is the least common denominator of $\beta_1,\ldots,\beta_l$.

Let $\chi_w$ be the character dual to $z w$. The group $G$ is a $Q$-principal bundle on $G/Q$ and the associated line bundle $L(w):=G\times^{\chi_w} \CC$ over $G/Q$ is ample. Note that there is a natural linearization of the $G$-action in $L(w)$ and we have $L(w)=i^* \mc{O}_{\PP(V)}(1)$. Since $\PB/Q$ is isomorphic to the associated projective bundle with fiber $G/Q$, we obtain a closed embedding $\PB/Q\to \PB_\sigma$, where $\sigma$ is the natural $G$-action on $V$. Thus, a parabolic principal bundle defines a  $\sigma$-decorated principal bundle. It is obvious that the corresponding natural transformation of moduli functors is injective.
\begin{definition}
 We call a parabolic principal bundle \emph{$\chi$-\sstable{}} if its associated $\sigma$-decorated principal bundle is $(\delta,\chi)$-\sstable{} with $\delta:=1/z$. We call two parabolic principal bundles \emph{S-equivalent} if their associated decorated principal bundles are so.
\end{definition}

For a $T$-split vector subbundle $F\subset E$ and a flag $U_\bullet$ of $E_{|x_0}$ we set
\[
 \pardeg^{U_\bullet}_{\tup{\beta},\tup{\chi}}(F):= \deg_{\tup{\ka},\tup{\chi}}(F)
 + \sum_{i=1}^{\len(U_\bullet)}\beta_i\dim_{\tup{\ka}}(U_i\cap F_{|x_0})
\]

\begin{lemma} 
\label{lem:stability_parabolic_PB}
 Let $(\PB,s)$ be a parabolic vector bundle, $E$ the $T$-split vector bundle associated with $\PB$ and $U_\bullet$ the flag of $E_{|x_0}$ defined by $s$ and $W_\bullet$. Then $(\PB,s)$ is $\chi$-\sstable{} if and only if the condition
 \[
  \sum_{j=1}^{\len(E_\bullet)} \alpha_j \left(\pardeg^{U_\bullet}_{\tup{\beta},\tup{\chi}}(E)\rk(E_j)-\pardeg^{U_\bullet}_{\tup{\beta},\tup{\chi}}(E_j)\rk(E) \right) (\ge ) 0
 \]
 holds for every weighted filtration coming from a weighted reduction of $\PB$.
\end{lemma}
\begin{proof}
 This is a consequence of Lemma \ref{lem:PB_stable<=>VB_stable} and the stability condition in flag varieties.
\end{proof}

\begin{corollary}
 For $\sum_{i=1}^l \beta_i<1$ the \textup{(}projective\textup{)} moduli space of $\chi$-\sstable{} parabolic principal bundles exists.
\end{corollary}
\begin{proof}
 The closed embedding $G/Q\to \PP(V)$ induces an injective natural transformation from the moduli functor of parabolic principal bundles to the moduli functor of decorated principal bundles. It also determines a closed invariant subscheme of the parameter space $P$ and hence a closed subscheme of the moduli space of decorated principal bundles, if it exists. Since $\sigma$ is polynomial and homogeneous of degree $\gamma:=z\sum_{i=1}^l \beta_i(r-r_i)$,
 Theorem \ref{thm:moduli_space_of_dec_PB} states the existence for $\gamma/z<1$. However, as explained in Remark 7.1 in \cite{Beck2014DecSwamps} in the case of parabolic bundles one can relax this condition to $\sum_{i=1}^l\beta_i<1$.
\end{proof}
\begin{remark}
 The moduli space of parabolic bundles with semisimple structure group was constructed in \cite{heinloth-schmitt}. By Lemma \ref{lem:stability_parabolic_PB} our stability condition specializes to their condition (compare Proposition 5.1.3 in \cite{heinloth-schmitt}). One checks that the notions of S-equivalence also agree. Thus, our result generalizes Theorem 3.2.3 in \cite{heinloth-schmitt} for reductive groups.
\end{remark}

\subsection{Principal Bundles with a Level Structure}
Let $G$ be a connected, semisimple group, $G_{\ad}:=G/Z(G)$ its adjoint group and $\pi:\tilde{G}\to G$ the universal cover. We fix a Borel subgroup $B\subset \tilde{G}$ and a maximal torus $T\subset B$. Denote by $\Phi\subset X^*(T)$ the set of roots of $G$, by $\Phi^+$ the set of positive roots with respect to $B$ and by $\Delta$ the set of simple positive roots. For every dominant weight $\chi\in X^*(T)$ there is an irreducible representation $V(\chi)$ of $\tilde{G}$ with highest weight $\chi$. We fix a \emph{regular} fundamental weight $\chi$, i.e., $\langle\alpha,\chi\rangle >0$ for all $\alpha \in \Delta$, and set $V:=V(\chi)$. The action of $\tilde{G}$ on $\PP(V^\vee)$ induces a faithful action of $G_{\ad}$, and the morphism
\[
 f:G_{\ad}\to \PP(\End(V)^\vee)
\]
is a $G\times G$-equivariant locally closed embedding of  $G_{\ad}$.
\begin{definition}
 The closure $\overline{G_{\ad}}:=\overline{\im(f)}$ is the \emph{wonderful compactification} of $G$.
\end{definition}
This definition is in fact independent of the choice of regular highest weight $\chi$.

\begin{proposition}[De Concini--Procesi, {\cite[\S 3.1, Theorem]{DCP83}}]
 The wonderful compactification has the following properties:
\begin{enumerate}
 \item $\overline{G_{\ad}}$ is smooth.
 \item $\overline{G_{\ad}}\setminus G_{\ad}$ is a union of smooth prime divisors $S_\alpha$, $\alpha\in \Delta$, with normal crossings.
 \item For every subset $I\subset \Delta$ the closed subscheme $S_I:=\cap_{\alpha\in I}S_\alpha$ is the closure of a unique orbit, and every orbit closure is of this form.
\end{enumerate}
\end{proposition}

For a subset $I\subset \Delta$ of simple roots let $\Phi_I$ denote the roots spanned by $\Delta\setminus I$. We consider the Lie-algebras
\[
 \mathfrak{l}_I:=\mathfrak{t}+ \sum_{\alpha\in \Phi_I}\mathfrak{g}_\alpha\,,
\]
as well as $\mathfrak{p}_I:=\mathfrak{l}_I+\mathfrak{u}$ and $\mathfrak{p}_I^-:=\mathfrak{l}_I+ \mathfrak{u}^-$. Here, $\mathfrak{t}$ is the Lie algebra of the torus $T$, $\mathfrak{u}$ is the Lie-algebra of the unipotent radical of $B$ and $\mathfrak{u}^-$ the Lie algebra of its opposite group. We denote the corresponding subgroups by $L_I$, $P_I$ and $P^-_I$.

There is finite set $\St_T(V)\subset X^*(T)$ and a decomposition of $V$ a direct sum of eigenspaces $V^{\chi'}$, $\chi'\in\St_T(V)$. We set 
\[
 \St_T^I(V):=\left\{\chi'\in \St_T(V)\,\biggm|\,\exists\,\tup{n}\in\NN^{\Delta\setminus I}: \chi'=\chi-\sum_{\alpha\in \Delta\setminus I} n_\alpha \alpha \right\}\,.
\]
and $V^I:=\bigoplus_{\chi'\in\St_T^I(V)} V^{\chi'}$. Denote by $\pr_I:V\to V^I$ the projection with $\pr_I(V^{\chi'})=0$ for $\chi'\notin \St_T^I(V)$.

\begin{proposition}[{\cite[\S 5.2, Theorem]{DCP83}}]
 \label{prop:geometry_of_orbits}
 Let $I\subset \Delta$. Then  $O_I$ is the $(G\times G)$-orbit of $[\pr_I]$ in $\PP(\End(V)^\vee)$. It is a fiber bundle over $G_{\ad}/P_I\times G_{\ad}/P_I^-$ with typical fiber $(L_I)_{\ad}$.
\end{proposition}
\begin{remark} \label{rem:wonderful_compactification}
 Let $\chi_\alpha$, $\alpha \in \Delta$, be dominant weights, such that their sum is regular. This is the case for the fundamental weights. Then one can construct the wonderful compactification also as the closure of the image of 
 \[
  G_{\ad}\to \prod_{\alpha\in \Delta}\PP(\End(V(\chi_\alpha))^\vee)\,.
 \]
 If the $\chi_\alpha$ are such that the $\tilde{G}$ action on $V(\chi_\alpha)$ comes from a representation of $G$, then
for a tuple of positive integers $\theta_\alpha$, $\alpha\in \Delta$, we obtain a natural linearization of the $G\times G$ action in the ample line bundle
\[
 \mc{O}_{\overline{G_{\ad}}}(\tup{\theta}):=\bigotimes_{\alpha\in\Delta}\pr^*_{\PP(\End(V(\chi_\alpha))^\vee)}\mc{O}_{\PP(\End(V(\chi_\alpha))^\vee)}(\theta_\alpha)
\]
\end{remark}

Let $\sigma$ denote the left action of $G\cong \{e\}\times G$ on $\overline{G_{\ad}}$.
\begin{definition}
A \emph{level structure} on a principal $G$-bundle $\PB$ is a point $s\in (\PB\times^\sigma \overline{G_{\ad}})_{|x_0}$.
\end{definition}
\begin{remark} \label{rem:level_str}
  By Proposition \ref{prop:geometry_of_orbits}, the datum of a level structure in $\PB$ corresponds to the datum of a subset $I\subset \Delta$, a reduction of $\PB_{|x_0}$ to $P^-_I$, a point in $G/P_I$ and a point in a fiber isomorphic to $(L_I)_{\ad}$.
\end{remark}

If $\chi$ is such that $V(\chi)$ is a representation of $G$, then there is a natural linearization of $\sigma$ in the ample line bundle $\mc{O}_{\PP(\End(V(\chi))^\vee}(1)$. A principal bundle with a level structure is a $\sigma$-decorated principal bundle. Since we assume $G$ to be semisimple, there are no non-trivial characters on $G$. The principal bundles with a level structure thus inherit a notion of $\delta_2$-(semi\nobreakdash-)stability.
\begin{corollary}
 For $a_2\delta_2<1$ the \textup{(}projective\textup{)} moduli space of $\delta_2$-\sstable{} principal bundles with a level structure exists.
\end{corollary}
\subsection{Stability of level structures}
While the existence of the moduli space of bundles with a level structure is a trivial consequence of the general theory, it is interesting to express the stability condition in terms of associated vector bundles for the classical groups.
\subsubsection{The special linear group}
We start by discussing the case $G=\SL(n)$. By Remark \ref{rem:wonderful_compactification}, the wonderful compactification $\overline{\PSL}(n)$ is the closure of the image of $\SL(n)$ in
\[
  \prod_{i=1}^{n-1} \PP\left(\End\left(\bigwedge^i \CC^r\right)^\vee\right)\,.
\]
By Remark \ref{rem:level_str}, a point $h$ in $\overline{\PSL}(n)$ corresponds to a sequence $1\le r_1< \ldots< r_k < r_{k+1}=n$, a descending flag $W_\bullet$ in $\CC^r$ of length $k$ with $\dim(W_j)=r-r_j$, an ascending flag $W'_\bullet$ in $\CC^r$ with $\dim(W'_j)=r_j$ and isomorphisms $[\alpha_j]\in \PP(\Hom(W_{j-1}/W_{j}, W'_j/W'_{j-1})^\vee)$ for $j=1,\ldots,k+1$.
For a number $1\le i \le n$ we set
\begin{align*}
 j_-(i):&=\max\{ j\,|\,j=0,\ldots,k\,,\, r_j<i\}\,, & i_-:&=r_{j_-(i)}\,,\\
 j^+(i):&=\min\{ j\,|\, j=1,\ldots,k\,,\, i\le r_j\}\,, & i^+:&=r_{j^+(i)} \,.
\end{align*}
According to \cite[Proposition 1]{lafforgue1998} the point $h$ is represented by the homomorphisms $h_i$ induced by the $\alpha_j$ via the following diagram:
 \begin{equation} \label{eq:vollst_Homomorphismus}
 \begin{gathered}
 \xymatrix{
  {\displaystyle \bigwedge^i \CC^r}  \ar@{->>}[r] \ar@{-->}[d]_{h_i} & {\displaystyle \left(\bigotimes_{j=1}^{j_-(i)} \bigwedge^{r_j-r_{j-1}}(W_{j-1}/W_j)\right) \otimes \bigwedge^{i-i_-} (W_{j_-(i)}/W_{j^+(i)}) } \ar[d] \\
   {\displaystyle \bigwedge^i \CC^n} & 
  {\displaystyle \left( \bigotimes_{j=1}^{j_-(i)} \bigwedge^{r_j-r_{j-1}} (W'_{j}/W'_{j-1}) \right)\otimes \bigwedge^{i-i_-} (W'_{j^+(i)}/W'_{j_-(i)}) }\rlap{\,.} \ar@{>->}[l]
  }
  \end{gathered}
 \end{equation}
 We fix an ample line bundle $\mc{O}_{\overline{\PSL}(n)}(\theta_1,\ldots,\theta_{n-1})$.
\begin{lemma}\label{lem:mu_in_wc}
 Let $\la\in X_*(\SL(n))$ be a one-parameter subgroup with associated weighted flag $(V_\bullet,\tup{\alpha})$ and $h\in \overline{\PSL}(n)$ with descending flag $W_\bullet$ as above. Then we have
 \[
  \mu(\la,h)=\sum_{j=1}^{\len(V_\bullet)} \alpha_j\sum_{i=1}^{n-1}\theta_i \left(n\, c_i(V_j,W_\bullet)-i\dim(V_j)\right)
 \]
with
\begin{equation} \label{EQ:definition_c}
 c_i(U,W_\bullet) :=\min\left\{\dim(U/(U\cap W_{j^+(i)})), \dim(U/(U\cap W_{j_-(i)}))+ i-i_-\right\}\,.
\end{equation}
\end{lemma}
\begin{proof}
 For fixed $j$ and $i$ one needs to determine the maximal dimension of a subspace $U$ of $V_j$, such that its basis vectors occur in a pure vector $v$ in $\bigwedge^i \CC^n$ with $h_i(v)\neq 0$. This is equivalent to $U\cap W_{j^+(i)}=\{0\}$ and $\dim(U\cap W_{j_-(i)})\le i-i_-$.
\end{proof}

Let now $(\PB,s)$ be an $\SL(n)$-bundle with a level structure. The bundle $\PB$ corresponds to a rank-$n$ vector bundle $E$ with trivial determinant. By Remark \ref{rem:level_str}, the datum of a level structure $s$ on $E$ determines a sequence $1\le r_1< \ldots< r_k < r_{k+1}=n$ and a descending flag $W_\bullet$ in $E_{|x_0}$ of length $k$ with $\dim(W_j)=r-r_j$. An easy consequence of Lemma \ref{lem:mu_in_wc} is the following:
\begin{proposition}
The $\SL(n)$-bundle $\PB$ with level structure $s$ is $\delta$-\textup{(}semi\nobreakdash-\textup{)}stable if and only if every non-trivial proper subbundle $F$ of the associated vector bundle $E$ satisfies
\[
   \deg(F)n \,(\le)\, \delta\sum_{i=1}^{n-1}\theta_i \left(c_i(F_{|x_0},W_\bullet)n -i\rk(F)\right)
\]
\end{proposition}
\begin{remark}
 This is just the stability condition for vector bundles with a level structure with trivial determinant (c.f. \cite[Th\'eor\`eme A]{Ngo2007} and \cite[Proposition 7.3]{Beck2014DecSwamps}).
\end{remark}

\subsubsection{The odd orthogonal group}
Now we consider the case $G=\SO(n,\CC)$ with $n=2r+1$. Let $\SO(n)\subset \GL(n)$ be the subgroup leaving invariant the symmetric bilinear form given by
\[
 M_n:= \begin{pmatrix}
  0 &  & 1 \\
   & \reflectbox{$\ddots$} & \\
  1 & & 0
 \end{pmatrix}\in \CC^{n\times n}\,.
\]
on $V:=\CC^n$. A maximal torus $T$ is then given by the diagonal matrices in $\SO(n)$. As a Borel subgroup we choose the upper triangular matrices in $\SO(n)$. We define the characters
\[
 L_i(\opname{diag}(z_1,\ldots,z_r,1,z_r^{-1},\ldots,z_{1}^{-1})):=z_i\,,\quad i=1,\ldots,r\,.
\]
 The simple roots are $\alpha_i:=L_i-L_{i+1}$ for $i=1,\ldots,r-1$ and $\alpha_r:= L_r$; the corresponding fundamental weights are $\omega_i:=\sum_{j=1}^j L_j$ for $i=1,\ldots, r-1$ and $\omega_r:=\frac{1}{2}\sum_{j=1}^r L_j$. In order to obtain representations of $\SO(n)$ rather than the spin group we work with the weights $\chi_i:=\omega_i$ for $i=1,\ldots, r-1$ and $\chi_r:=2\omega_r$. The corresponding irreducible representations are just $V(\chi_i)=\bigwedge^i V$ (see e.g. \cite[\S 18]{Fulton-Harris}). 
According to Remark \ref{rem:wonderful_compactification} the wonderful compactification $\overline{\SO}(n)$ can be constructed inside
\[
 Y:=\prod_{i=1}^{r} \PP\left(\End\left(\bigwedge^i \CC^{2r+1}\right)^\vee\right)\,.
\]
A point $h\in \overline{\SO}(n)$ corresponds to a subset $\tup{r}\subset \{1,\ldots,r\}$ and a descending coisotropic flag $W_\bullet$ of length $k=|\tup{r}|$ in $\CC^n$ with $\dim(W_j)=n-r_j$, an ascending isotropic flag $W'_\bullet$ in $\CC^n$ with $\dim(W'_j)=r_j$ and isomorphisms $[\alpha_j]\in\PP(\Hom(W_{j-1}/W_i, W'_{j-1}/W'_i)^\vee)$ for $j=1,\ldots,k+1$. Here we use $W_{k+1}:=W_k^\bot$ and $W'_{k+1}=W'^\bot_k$. Note that $\alpha_{k+1}$ is orthogonal. 

The point $h$ can reconstructed from these data via the diagram \eqref{eq:vollst_Homomorphismus}.
We now fix an ample line bundle $\mc{O}_{\overline{\SO}(n)}(\theta_1,\ldots,\theta_{r})$.
\begin{lemma} \label{lem:odd_orth}
 Let $\la\in X_*(\SO(n))$ be a one-parameter subgroup with associated weighted flag $(V_\bullet,\tup{\alpha})$ and $h\in \overline{\SO}(n)$ with descending coisotropic flag $W_\bullet$ as above. Then we have
 \[
  \mu(\la,h)=\sum_{j=1}^{\len(V_\bullet)/2} \alpha_j
    \sum_{i=1}^{n-1}\theta_i n c'_i(V_j,W_\bullet)
 \]
with
\[
  c'_i(U,W_\bullet):=c_i(U,W_\bullet)+c_i(U^{\bot},W_\bullet)-i\,.
 \]
\end{lemma}
\begin{proof}
 The one-parameter subgroup $\lambda$ determines a weighted flag of the form
 \[
  \{0\}\subset V_1\subset \cdots \subset V_s\subset V_s^\bot\subset \cdots \subset V_1^\bot\subset \CC^n
 \]
with isotropic spaces $V_1,\ldots,V_s$ and weights $(\alpha_1,\ldots,\alpha_s,\alpha_s,\ldots,\alpha_1)$. Since $\dim(U^\bot)=n-\dim(U)$ the formula follows from Lemma \ref{lem:mu_in_wc}.
\end{proof}

Let $(\PB,s)$ be an $\SO(n)$-bundle with a level structure. The bundle $\PB$ determines a vector bundle $E$ of rank $n$ together with a symmetric isomorphism $\phi:E\to E^\vee$. A level structure on $\PB$ determines a subset $\tup{r}\subset \{1,\ldots,r\}$ and a descending coisotropic flag $W_\bullet$ of length $k=|\tup{r}|$ in $E_{x_0}$ with $\dim(W_j)=n-r_j$.
\begin{proposition} \label{prop:odd_orthogonal_group}
 The $\SO(n)$-bundle $\PB$ with a level structure is $\delta$-\textup{(}semi\nobreakdash-\textup{)}stable if and only if every isotropic subbundle $F\subset E$ satisfies
 \[
  2\deg(F) (\le) \delta\sum_{i=1}^{r}\theta_i  c'_i(F_{|x_0},W_\bullet)\,.
 \]
 \end{proposition}
\begin{proof}
 The test objects are weighted reductions of $\PB$ to a parabolic subgroup, i.e. flags of the form
 \begin{equation}\label{eq:weighted_isotropic_flag}
   \{0\}\subset F_1\subset \cdots \subset F_l\subset F_l^\bot \subset \cdots \subset F_1^\bot \subset E\,
 \end{equation}
Since $\deg(F_j)=\deg(F_j^\bot)$ and $\deg(E)=0$ the claim follows from Lemma \ref{lem:odd_orth}.
\end{proof}

\subsubsection{The symplectic group}
We come to the case $G=\Sp(2r)$. Let $\Sp(2r)\subset \GL(2r)$ be the subgroup fixing the antisymmetric bilinear form 
\[
 J=\begin{pmatrix}
    0 & M_r\\
    -M_r & 0
   \end{pmatrix} 
\]
on $V:=\CC^{2r}$. With this convention, we can choose the set of diagonal matrices as a maximal torus $T$ and the set of upper triangular matrices as a Borel subgroup. We define the characters $L_i\in X^*(T)$ by
\[
 L_i(\opname{diag}(z_1,\ldots,z_r,z_r^{-1},\ldots,z_1^{-1})):=z_i\,.
\]
Then, the primitive roots are $\alpha_i:=L_i-L_{i+1}$ for $1\le i \le r-1$ and $2L_r$. The corresponding fundamental weights are $\omega_i:=\sum_{j=1}^j L_j$ for $i=1,\ldots, r$. Let $\phi_i:\bigwedge^i V\to \bigwedge^{i-2}V$ be the contraction using $J$. Then, the irreducible representations with highest weight $\omega_i$ are $V_1:=V$ for $i=1$ and $V_i:=\ker(\phi_i)$ for $i=2,\ldots,r$ (see \cite[\S 17]{Fulton-Harris}). Note that a pure vector in $\bigwedge^i V$ lies in $V_i$ if and only if it is a wedge product of basis vectors of an isotropic subspace. The wonderful compactification $\overline{\Sp}(2r)$ of $\Sp(2r)$ can thus be constructed inside
\[
 \prod_{i=1}^r \PP\left(\End(V_i)^\vee\right)\,.
\]
A point $h\in \overline{\Sp}(2r)$ determines a subset $\{r_1,\ldots,r_k\}\subset \{1,\ldots,r\}$, a descending coisotropic flag $W_\bullet$ of length $k$ in $V$ with $\dim(W_j)=2r-r_j$, an ascending isotropic flag $W'_\bullet$ in $V$ with $\dim(W'_j)=r_j$, isomorphisms $\alpha_j\in\PP(\Hom(W_{j-1}/W_{j}, W'_j/W'_{j-1})^\vee)$ for $j=1,\ldots,k$ and a symplectic isomorphism $[\alpha_{k+1}]\in\PP(\Hom(W_k/W_k^\bot , W'^\bot_k/W'_k)^\vee)$. The homomorphisms $h_i$ defined as in \eqref{eq:vollst_Homomorphismus} restrict to endomorphisms of $V_i$, which represent $h$.

We fix an ample line bundle on $\overline{\Sp}(2r)$ by choosing positive integers $\theta_1,\ldots,\theta_r$.
\begin{lemma}
 Let $\la\in X_*(\Sp(2r))$ be a one-parameter subgroup with associated weighted flag $(V_\bullet,\tup{\alpha})$ and $h\in \overline{\Sp}(2r)$ with descending coisotropic flag $W_\bullet$ as above. Then we have
 \[
  \mu(\la,h)=2r\sum_{j=1}^{\len(V_\bullet)/2}\alpha_j \sum_{i=1}^r \theta_i 
 \max\{c_i'(V_j,\tilde{W}_\bullet)\,|\,\tilde{W}_\bullet\in D(W_\bullet)\}\,.
 \]
 Here, $D(W_\bullet)$ denotes the set of all extensions $\tilde{W}_\bullet:W_1\supset \cdots \supset W_k\supset W$ of $W_\bullet$ such that $W$ is maximal isotropic.
\end{lemma}
\begin{proof}
The flag $V_\bullet$ is of the form
\[
 \{0\}\subset V_1\subset \cdots \subset V_s\subseteq V_{s+1}=V_s^{\bot} \subset \cdots \subset V_{2s}=V_1^\bot\subset V
\]
with isotropic subspaces $V_1,\ldots,V_s$ and weights $(\alpha_1,\ldots,\alpha_s,\alpha_s,\ldots,\alpha_1)$. Now for fixed $i$ and $j$ one needs to determine a subspace $U\subset V_j$ of maximal dimension, such its basis vectors occur in a wedge product $v\in V_i$ with $h_i(v)\neq 0$. The former condition means that $U$ is isotropic. Now one can always choose a symplectic basis of $V$ that extends bases of $W_\bullet$ and $V_j$. Then it is clear that for such a subspace $U$ there exists an isotropic subspace $W\subset W_k$ with $U\cap W=\{0\}$. Conversely, if $W\subset W_k$ is maximal isotropic, then there exists an isotropic subspace $U\subset V_j$ with $U\cap W=\{0\}$ and $\dim(U\cap W_{j_-(i)})\le i-i_-$ of dimension $c(V_j,\tilde{W}_\bullet)$.
\end{proof}

Let now $\PB$ be a principal $\Sp(2r)$-bundle with a level structure $s$.
The datum of $\PB$ is equivalent to the datum of a vector bundle $E$ of rank $2r$ with an antisymmetric isomorphism $\phi:E\to E^\vee$. The level structure determines a subset $\{r_1,\ldots,r_k\}\subset \{1,\ldots,r\}$ and a descending coisotropic flag $W_\bullet$ of length $k$ in $E_{|x_0}$ with $\dim(W_j)=2r-r_j$.
\begin{proposition} \label{prop:symplectic}
The symplectic bundle $\PB$ with level structure $s$ is $\delta$-\textup{(}semi\nobreakdash-\textup{)}stable if and only if every non-trivial proper isotropic subbundle $F$ of the associated vector bundle $E$ satisfies
  \begin{align*}
  2\deg(F) (\le)\, \delta\sum_{i=1}^r \theta_r \max\{c_i'(F_{|x_0},\tilde{W}_\bullet)\,|\,\tilde{W}_\bullet\in D(W_\bullet)\}\,.
 \end{align*}
\end{proposition}

\subsubsection{The even orthogonal group}
Finally, let $\SO(2r)\subset \GL(2r)$ be the subgroup fixing the symmetric bilinear form given by $M_{2r}$ on $V:=\CC^{2r}$. Here, the discussion is complicated by the presence of two types of maximal isotropic subspaces, the self-dual and the anti-self-dual ones. The isomorphism $\bigwedge^{2r}V\to \CC$ defined by $e_1\wedge\ldots\wedge e_{2r}\mapsto 1$ defines a non-degenerate bilinear from $\wedge:\bigwedge^r V\times \bigwedge^r V\to \CC$ and hence an isomorphism $\psi:\bigwedge^rV\to \bigwedge^r V^\vee$. The map $\tau:=\psi^{-1}\circ \wedge^r J$ satisfies $\tau^2=\id$ and the representation $\bigwedge^r V$ decomposes into invariant eigenspaces $V^{\pm}$ with eigenvalues $\pm 1$. The image of a subspace $U$ of $V$ under the Pl\"ucker embedding lies in the projectivization of one of these if and only if $U$ is maximal isotropic. One says that $U$ is \emph{self-dual} or of type $+1$ if its image lies in $\PP(V^+)$ and \emph{anti-self-dual} or of type $-1$ if its image lies in $\PP(V^-)$. Note that if 
$U$ and $W$ are maximal isotropic subspaces of type $a$ and $b$ respectively, then $a=b(-1)^{\dim(U\cap W)-r}$.

Let $T$ be the maximal torus of diagonal matrices and $B$ the Borel subgroup of upper triangular matrices. We define the characters $L_i$ in $X^*(T)$ by
\[
 L_i(\opname{diag}(z_1,\ldots,z_r,z_r^{-1},\ldots,z_1^{-1})):=z_i\,, \qquad 1\le i \le r\,.
\]
The simple roots are $\alpha_i:=L_i-L_{i+1}$ and $\alpha_r:=L_{r-1}+ L_r$; the corresponding fundamental weights are $\omega_i:=\sum_{j=1}^i L_j$ for $i=1,\ldots,r-2$ and $\omega_{\pm}:=\frac{1}{2}(\sum_{j=1}^{r-1} L_j\pm L_r)$. We work with the weights $\chi_i:=\omega_i$ and $\chi_{\pm}:=2\omega_{\pm}$. Their irreducible representations are $V(\chi_i)=\bigwedge^i V$ for $i=1,\ldots, r-2$ and $V(\chi_\pm)= V^\pm$ (see \cite[\S 19]{Fulton-Harris}). The wonderful compactification $\overline{\SO}(2r)$ is the closure of $\SO(2r)$ in
\[
 \PP(\End(V_r^+)^\vee)\times \PP(\End(V_r^-)^\vee)\times  \prod_{i=1}^{r-2}\PP(\End(\bigwedge^i V)^\vee)\,.
\]
A point $h\in \overline{\SO}(n)$ determines the following: A subset $I$ of $\{1,\ldots,r-2,+,-\}$, a descending coisotropic flag $W_\bullet$ of length $k$ in $V$ with $\dim(W_j)=2r-r_j$, an ascending isotropic flag $W'_\bullet$ in $V$ with $\dim(W'_j)=r_j$ and isomorphisms $[\alpha_j]\in\PP(\Hom(W_{j-1}/W_j, W'_j/W'_{j-1})^\vee)$ for $j=1,\ldots,k$.

Furthermore, if $\{+,-\}\cap I=\varnothing$ there is an orthogonal isomorphism 
\[
 [\alpha_{k+1}]\in \PP(\Hom(W_k/W^\bot_k, W'^\bot_k/W'_k)^\vee)\,.
\]
For $\pm \in I$ there is a maximal isotropic subspace $W_\pm\subset W_k$ of type $\pm(-1)^r$ and a maximal isotropic subspace $W'_{\pm}\supset W'_k$ of type $\pm$. If $\pm\in I$ and $\mp\notin I$, then there is an isomorphism 
\[
[\alpha_\pm]\in\PP(\Hom(W_{k}/W_\pm,W'_{\pm}/W'_{k})^\vee)\,.
\]
If we have $\{+,-\}\subset I$, then $W_{k+1}:=W_+ +W_-$ is coisotropic of dimension $r+1$, $W'_{k+1}:=W'_+ \cap W'_-$ is isotropic of dimension $r-1$ and there are isomorphisms $[\alpha_{k+1}]\in\PP(\Hom(W_{k}/W_{k+1},W'_{k+1}/W'_k)^\vee)$ and  $[\alpha_\pm]\in\PP(\Hom(W_{k+1}/W_\pm,W'_{\pm}/W'_{k+1})^\vee)$. Note that in this case there are canonical isomorphisms
\begin{align*}
 (W_{k+1}/W_+)^\vee &\cong (W_+/(W_+\cap W_-))\cong W_{k+1}/W_-\,,\\
 (W'_{+}/W'_{k+1})^\vee &\cong (W'_+ + W'_-)/W'_+ \cong W'_-/W'_{k+1}\,,
\end{align*}
and one has $[(\alpha_+^\vee)^{-1}]=[\alpha_-]$ in $\PP(\Hom(W_{k+1}/W_\pm,W'_{\pm}/W'_{k+1})^\vee)$.

We have to explain how to reconstruct $h$ from the data above: For $i=1,\ldots,r-2$ we define homomorphisms $h_i$ as in \eqref{eq:vollst_Homomorphismus}.
If $\{+,-\}\cap I=\varnothing$, the homomorphism $h_r$ defined as in \eqref{eq:vollst_Homomorphismus} using $W_{k+1}:=W_k^\bot$ and $W'_{k+1}:=W'^\bot_k$ restricts to homomorphisms $h_\pm:E_{|x_0}^{\pm}\to V^\pm$ because $\alpha_{k+1}$ is orthogonal. If $\{+,-\}\subset I$, then one defines $h_+$ and $h_-$ via \eqref{eq:vollst_Homomorphismus} with the flags $W_1\supset \cdots \supset (W_+ + W_-)\supset W_\pm$ and $W'_1\subset \cdots \subset W'_{+}\cap W'_-\subset W'_\pm$. 

If $\{+,-\}\cap I= \{+\}$, $h_+$ can be constructed via \eqref{eq:vollst_Homomorphismus} using $W_{k+1}:=W_+$, $W'_{k+1}:=W'_+$ and $\alpha_{k+1}=\alpha_+$. In order to define $h_-$, we consider the isomorphism
\[
(\alpha_+^\vee)^{-1}:W_+/W_k^\bot\cong (W_k/W_+)^\vee\to (W'_+/W'_k)^\vee\cong W'^\bot_k/W'_+\,.
\]
 Now, $h_-$ is defined as the homomorphism induced by $\alpha_1,\ldots\alpha_k,\alpha_+$ and $(\alpha_{+}^{\vee})^{-1}$ via
\[
 \xymatrix{
  {\displaystyle \bigwedge^r V}  \ar@{->>}[r] \ar@{-->}[d]_{h_-} & {\displaystyle \left(\bigotimes_{j=1}^{k} \bigwedge^{r_j-r_{j-1}}(W_{j-1}/W_j)\right)\otimes   \left(\bigwedge^{r-r_k-1} W_k/W_+\right) \otimes W_+/W_k^\bot } \ar[d] \\
   {\displaystyle \bigwedge^r V} & 
  {\displaystyle \left(\bigotimes_{j=1}^{k} \bigwedge^{r_j-r_{j-1}} (W'_{j}/W'_{j-1})\right) \otimes \left(\bigwedge^{r-r_k-1} W'_+/W'^\bot_k\right) \otimes W'^\bot_k/W'_+ }\rlap{\,.} \ar@{>->}[l]
  }
 \]
A similar construction yields $h_+$ if $\{+,-\}\cap I=\{-\}$. These homomorphisms represent $h$.

We define $D_{\pm}(W_\bullet)$ to be the set of flags $W_1\supset \cdots \supset W_{k}\supset W$  with the following properties:
\begin{itemize}
 \item If $\pm\in I$, then $W=W_\pm$.
 \item In any case, $W$ is maximal isotropic of type $\pm(-1)^r$.
 \item If $\mp\in I$, then $\dim(W+ W_{\mp})=r+1$.
\end{itemize}

\begin{lemma} \label{lem:even_orth}
 Let $\la\in X_*(\SO(2r))$ be a one-parameter subgroup with associated weighted flag $(V_\bullet,\tup{\alpha})$ and $h\in \overline{\SO}(2r)$ with descending coisotropic flag $W_\bullet$ as above. Then we have
 \[
  \mu(\la,[h_\pm])=\max\left\{c'_r(U,\tilde{W}_\bullet)\,\middle|\,\tilde{W}_\bullet\in D_{\pm}(W_\bullet) \right\}=: c'_{\pm}(U,W_\bullet)\,.
 \]
\end{lemma}
\begin{proof}
 Since any two parabolic subgroups contain a common maximal torus, we may assume that $V_\bullet$ and $W_\bullet$ are defined using the standard basis of $V$. If $\{+,-\}\cap I=\varnothing$, then the argument is the same as in the symplectic case. If $+ \in I$, the set $D_+(W_\bullet)$ contains precisely one flag and the result is obvious. Similarly if $-\in I$. Suppose now that $\{+,-\}\cap I=\{+\}$. Then for a subspace $U\subset V_j$ contained in an anti-self-dual maximal isotropic subspace $U'$ with $\dim(U'\cap W_+)=1$ and $U'\cap W_k^\bot=\{0\}$ there exists a maximal isotropic complement $W\subset W_k$ of $U'$ of type $-(-1)^r$ with $\dim(W+W_+)=r+1$. Conversely, if $W\subset W_k$ is a maximal isotropic subspace of type $-(-1)^r$ with $\dim(W+W_+)=r+1$, then there is a subspace $U\subset V_j$ of dimension $c(V_j,\tilde{W}_\bullet)$ that is contained in an anti-self-dual maximal isotropic subspace $U'\subset V$ with with $\dim(U'\cap W_+)=1$ and $U'\cap W_k^\bot=\{0\}$.
\end{proof}

Let $\PB$ be an $\SO(2r)$-bundle with a level structure. The bundle $\PB$ determines a vector bundle $E$ of rank $2r$ together with a symmetric isomorphism $\phi:E\to E^\vee$. A level structure on $\PB$ provides a subset $I$ of $\{1,\ldots,r-2,+,-\}$ and a descending coisotropic flag $W_\bullet$ of length $k$ in $E_{|x_0}$ with $\dim(W_j)=2r-r_j$. We fix an ample line bundle on $\overline{\SO}(2r)$ by choosing positive integers $\theta_1,\ldots,\theta_{r-2},\theta_-$ and $\theta_+$.
\begin{proposition}
 The $\SO(2r)$-bundle $\PB$ with a level structure is $\delta$-\textup{(}semi\nobreakdash-\textup{)}stable if and only if every non-trivial proper isotropic subbundle $F$ of the associated vector bundle $E$ satisfies
 \begin{align*}
  2\deg(F) \,(\le)\, \delta\theta_{+}c'_+ (F_{|x_0},W_\bullet) +\delta\theta_{-}c'_- (F_{|x_0},W_\bullet) +\delta\sum_{i=1}^{r-2}\theta_i c'_i(F_{|x_0},W_\bullet)  \,.
\end{align*}
\end{proposition}
\begin{proof}
 The computation of $\mu(\la,[h_i])$ for $i=1,\ldots,r-2$ has been carried out in the proof of Lemma \ref{lem:mu_in_wc}. Together with Lemma \ref{lem:even_orth} the claim follows.
\end{proof}

\bibliographystyle{amsplain}

\end{document}